\documentclass[11pt]{amsart}

\usepackage{amssymb}
\usepackage{amsmath}
\usepackage{amsfonts}
\usepackage{graphicx}
\usepackage{subfigure}
\usepackage{amsthm}
\usepackage{enumerate}
\usepackage[mathscr]{eucal}
\usepackage{mathrsfs}
\usepackage{verbatim}
\usepackage{yhmath}
\usepackage{epstopdf}
\usepackage{color}
\usepackage[colorlinks,linkcolor=red,anchorcolor=green,citecolor=blue]{hyperref}

\makeatletter
\@namedef{subjclassname@2020}{%
  \textup{2020} Mathematics Subject Classification}
\makeatother

\numberwithin{equation}{section}
\numberwithin{figure}{section}

\theoremstyle{plain}
\newtheorem{theorem}{Theorem}[section]
\newtheorem{lemma}[theorem]{Lemma}
\newtheorem{proposition}[theorem]{Proposition}
\newtheorem{corollary}[theorem]{Corollary}

\theoremstyle{definition}
\newtheorem{defn}[theorem]{Definition}
\newtheorem{defnprop}[theorem]{Definitions and Basic Properties}

\newtheorem{example}[theorem]{Example}

\theoremstyle{plain}

\theoremstyle{remark}
\newtheorem*{remark}{Remark}

\allowdisplaybreaks[4]

\begin{document}
\date{}

\title[A multidimensional Szemer\'{e}di theorem]{A multidimensional Szemer\'{e}di theorem in integers}

\author{Jingwei Guo}
\address{School of Mathematical Sciences\\
	University of Science and Technology of China\\
	Hefei, 230026\\ P.R. China}
\email{jwguo@ustc.edu.cn}

\author{Changxing Miao}
\address{Institute of Applied Physics \& Computational Mathematics\\
	Beijing, 100088\\P.R. China}
\email{miao\_changxing@iapcm.ac.cn}

\author{Guoqing Zhan}
\address{School of Mathematical Sciences\\
	University of Science and Technology of China\\
	Hefei, 230026\\ P.R. China}
\email{zhanguoqing@mail.ustc.edu.cn}

\thanks{}



\begin{abstract}
For any integer $n \geq 2$, let $(m_{1},\ldots,m_{n})$ be a strictly increasing $n$-tuple of positive integers. We show that any subset $A\subset [N]^n$ of density at least $(\log N)^{-c}$ contains a nontrivial configuration of the form
\begin{equation*}
\boldsymbol{x},\boldsymbol{x}+r^{m_{1}}\boldsymbol{e_{1}},\ldots,\boldsymbol{x}+r^{m_{n}}\boldsymbol{e_{n}},
\end{equation*}
where $c=c(n,m_{1},\ldots,m_{n} )$ is a  positive constant. This quantitative multidimensional Szemer\'{e}di theorem  extends a recent two-dimensional result of Peluse, Prendiville, and Shao concerning the configuration of the form $(x,y),(x+r,y),\left(x,y+r^{2}\right)$.  The theorem is obtained as a consequence of an effective ``popular'' version.
\end{abstract}

\maketitle

\tableofcontents

\section{Introduction} \label{intro}

Let $[N]$ denote the integer set $\{1, \dots, N\}$. In 1954, Roth \cite{R54} proved that if  $A\subset [N]$ contains no nontrivial three-term arithmetic progressions, then
\begin{equation*}
|A|=O\left(N \left(\log\log N\right)^{-1}\right).
\end{equation*}
This theorem of Roth initiated a major line of research, with important improvements made by  Heath-Brown, Szemer\'{e}di, Bourgain, Sanders, Bloom, Schoen, and Bloom--Sisask. A recent breakthrough in 2023 by Kelley and Meka \cite{KM23} improved the bound to
\begin{equation*}
|A|\leq N \exp\left(-c(\log N)^{1/12} \right)
\end{equation*}
for some absolute constant $c>0$. The exponent $1/12$ was later improved to $1/9$ by Bloom and Sisask \cite{BS23}. For a historical account on this problem as well as further details, we refer the interested reader to \cite{KM23} and the references therein.

A central direction in additive combinatorics seeks to extend Roth's result by proving upper bounds for sets that avoid $k$-term arithmetic progressions (for any fixed $k\geq 4$). This line of inquiry was established by two foundational works: Szemer\'{e}di's seminal theorem \cite{S75}, which established the bound $|A|=o(N)$ for such sets, and Gowers's subsequent work \cite{G98, G01}, which introduced a deep analytical framework that revolutionized the approach to these and related problems and yielded the quantitative improvement
\begin{equation*}
|A|=O\left( N(\log\log N)^{-c_k}\right)
\end{equation*}
for some constant $c_k>0$. A further improvement was then achieved, in the case $k=4$, by Green and Tao \cite{GT09, GT17}, who showed that
\begin{equation*}
	|A|=O\left( N(\log N)^{-c} \right)
\end{equation*}
for some constant $c>0$. In 2024, Leng, Sah, and Sawhney \cite{LSS24-1, LSS24-2} achieved a breakthrough, obtaining for all $k \geq 5$ the improved bound
\begin{equation*}
	|A|=O\left( N\exp\left(-(\log \log N)^{c_k}\right) \right)
\end{equation*}
with $c_k \in (0,1)$. A key factor behind this advance was the introduction of quantitative refinements to the theory of nilsequences, which demonstrates its power in this area of study.

The classical Szemer\'{e}di theorem, as well as its multidimensional generalization by Furstenberg and Katznelson \cite{FK78}, deals with linear configurations such as arithmetic progressions. Many problems, however, involve nonlinear patterns; for instance, progressions whose increments are given by polynomials. One famous result in this direction is the polynomial Szemer\'{e}di theorem of Bergelson and Leibman \cite{BL96}. It states that for any integer polynomials $P_1, \ldots, P_m\in \mathbb{Z}[y]$  with zero constant term, any subset of $[N]$ that contains no progression of the form
\begin{equation*}
	x, x+P_1(y), \ldots, x+P_m(y) \quad \textrm{with $y\in\mathbb{Z}\setminus \{0\}$}
\end{equation*}
must have size $o(N)$. In fact, results in Bergelson and Leibman \cite{BL96} also include more general higher‑dimensional generalizations. A major quantitative challenge in additive combinatorics is to establish quantitatively effective versions of theorems of Bergelson and Leibman, especially when the polynomials involved have different degrees.  Below we focus on (mainly nonlinear) configurations of length at least three. For results concerning linear configurations or configurations of length two, we refer the reader to the references listed in \cite{PP24, PPS24}.

In this direction, there has been significant progress in recent years. Prendiville \cite{Pre17} considered homogeneous polynomials of the same degree  and proved that if $A\subset [N]$ contains no nontrivial progression of the form
\begin{equation*}
	x, x+c_1 y^k, \ldots, x+c_m y^k
\end{equation*}
for $c_1, \ldots, c_m\in\mathbb{Z}$, then
\begin{equation*}
	|A|=O\left(N(\log\log N)^{-c_{m,k}} \right)
\end{equation*}
for some constant  $c_{m,k}>0$.  In a breakthrough result, Peluse and Prendiville \cite{PP24} obtained the first bound for a progression involving polynomials of distinct degrees. They showed that if $A\subset [N]$ contains no nontrivial progression of the form
\begin{equation*}
	x, x+y, x+y^2,
\end{equation*}
then
\begin{equation*}
|A|=O\left(N(\log\log N)^{-c} \right)
\end{equation*}
for some constant  $c>0$. Subsequently, in \cite{PP22}, they improved this bound to
\begin{equation*}
|A|=O\left(N(\log{N})^{-c}\right)
\end{equation*}
for some constant  $c>0$. The case of general polynomial progressions where the polynomials have distinct degrees has been studied in Peluse \cite{Pel20} and later in Shao and Wang \cite{SW25}.

In the two‑dimensional setting,  Peluse, Prendiville, and Shao \cite{PPS24} proved that if $A\subset [N]^2$ contains no nontrivial triple of the form
\begin{equation*}
(x, y), (x+d,y), (x, y+d^2),
\end{equation*}
then
\begin{equation*}
|A|=O\left(N^2(\log N)^{-c}\right)
\end{equation*}
for some constant $c>0$. Their work actually provides an effective ``popular'' version of this result and a new proof of the  aforementioned results of Peluse and Prendiville \cite{PP24, PP22}. Separately, Kravitz, Kuca, and Leng \cite{KKL24 a} showed that if $A\subset [N]^2$ does not contain any configurations of the form
\begin{equation*}
	(x, y), (x+P(d),y), (x, y+P(d)),
\end{equation*}
where $P\in \mathbb{Z}[d]$ has an integer root of multiplicity one, then
\begin{equation*}
	|A|=O_P\left( N^2(\log\log\log N)^{-c} \right)
\end{equation*}
for some constant $c>0$. To the best of our knowledge, these two results are the only known effective multidimensional polynomial Szemer\'{e}di theorems over the integers that involve a nonlinear polynomial. Related analogs exist elsewhere: in finite fields both for the  configuration $(x,y)$, $(x+P_1(d),y)$, $(x, y+P_2(d))$ (see Han, Lacey, and Yang \cite{HLY21}) and for its multidimensional polynomial generalizations (see Kuca \cite{Kuc24,Kuc24-b}), and in the Euclidean setting $\mathbb{R}^2$  both for the pattern $(x,y)$, $(x+t,y)$, $(x, y+t^2)$ (see Christ, Durcik, and Roos \cite{CDR21}) and for its two-polynomial generalizations (see Chen with the first- and second-named authors, respectively \cite{CG24, CM24}).

The main result of this paper is a multidimensional Szemer\'{e}di theorem in $\mathbb{N}^n$ with $n\geq 2$. It extends  a   two-dimensional result of Peluse, Prendiville, and Shao \cite{PPS24}.
Let $\boldsymbol{e}_{\boldsymbol{1}}$, $\boldsymbol{e}_{\boldsymbol{2}}$, \ldots,  $\boldsymbol{e}_{\boldsymbol{n}}$ denote the standard basis vectors of $\mathbb{R}^n$.

\begin{theorem} \label{main theorem}
For any integer $n \geq 2$, let $(m_{1},\ldots,m_{n})$ be  a strictly increasing $n$-tuple of positive integers,
and  $N_{1},\ldots,N_{n}\in \mathbb{N}$ satisfy
\begin{equation}\label{s1-2}
	N_{n}^{1/m_{n}}\leq N_{n-1}^{1/m_{n-1}}\leq\cdots\leq N_{1}^{1/m_{1}}.
\end{equation}
 If $N_{n}$ is sufficient large and $A\subset \prod_{1\leq i\leq n} \left[N_{i}\right]\subset \mathbb{N}^n$ contains no nontrivial configuration of the form
\begin{equation}\label{s1-1}
	\boldsymbol{x}, \boldsymbol{x}+r^{m_{1}} \boldsymbol{e}_{\mathbf{1}}, \ldots, \boldsymbol{x}+r^{m_{n}} \boldsymbol{e}_{\boldsymbol{n}},
\end{equation}	
then
\begin{equation*}
	|A|\leq (\log N_{n})^{-c} \prod_{1\leq i\leq n}N_{i}
\end{equation*}
for some constant $c=c\left(n,m_{1},\ldots,m_{n}\right)>0$.

\end{theorem}

\begin{remark}
Two direct consequences of the theorem, obtained by choosing specific $N_i$, are worth stating separately.
\begin{enumerate}[(i)]
\item  By taking $N=N_1=N_2=\cdots=N_n$, we obtain:  if $N$ is sufficient large and $A\subset [N]^n$ contains no nontrivial configuration \eqref{s1-1} with a strictly increasing $n$-tuple $\left(m_{1},\ldots,m_{n}\right)$ of positive integers, then
\begin{equation*}
|A|\leq N^n(\log N)^{-c}
\end{equation*}
for some constant $c>0$.

\item  By taking $N_i=N^{m_i}$, $1\leq i\leq n$, we obtain:  if  $N$ is sufficient large and $A\subset \prod_{1\leq i\leq n} \left[N^{m_{i}}\right]$ contains no nontrivial configuration \eqref{s1-1} with a strictly increasing $n$-tuple $\left(m_{1},\ldots,m_{n}\right)$  of positive integers, then
\begin{equation*}
	|A|\leq N^{m_1+\cdots+m_n}(\log N)^{-c}
\end{equation*}
for some constant $c>0$.
\end{enumerate}
\end{remark}

As in \cite{PPS24}, we derive Theorem \ref{main theorem} from the following stronger result concerning the existence of a ``popular'' difference.

\begin{theorem}\label{popular difference,n-D}
Let $\delta\in(0,1/10)$, $(m_{1},\ldots,m_{n})$ be a strictly increasing $n$-tuple of positive integers, and $N_{1},\ldots,N_{n}\in\mathbb{N}$ satisfy \eqref{s1-2}.	If $N_{n}\geq \exp (\delta^{-C})$ for some constant  $C=C(n,m_{1},\ldots,m_{n})>0$, then for any set $A\subset \prod_{1\leq i\leq n} \left[N_{i}\right]\subset \mathbb{N}^n$, there exists $r\geq 1$ such that 	
	\begin{equation*}
		\left|\left\{\boldsymbol{x} \in A: \boldsymbol{x}+r^{m_{1}}\boldsymbol{e_{1}}, \ldots, \boldsymbol{x}+r^{m_{n}} \boldsymbol{e}_{\boldsymbol{n}} \in A\right\}\right| \gtrsim_{n} \left( \mu_{A}^{n+1}- \delta\right)\prod_{1\leq i\leq n}N_{i},
	\end{equation*}
	where $\mu_{A}:=|A|/\prod_{1\leq i\leq n}N_{i}$.
\end{theorem}

As an easy one-dimensional corollary of Theorem \ref{popular difference,n-D}, we have the following.

\begin{corollary}\label{popular difference,1-D}
	Let $\delta\in(0,1/10)$, $n\geq 2$, and $\left(m_{1},\ldots,m_{n}\right)$ be a strictly increasing $n$-tuple of positive integers. If $N\geq \exp (\delta^{-C})$ for some constant $C=C(n,m_{1},\ldots,m_{n})>0$, then for any set $A\subset [N]$, there exists $r\geq 1$ such that 	
	\begin{equation*}
		\left|\left\{x \in A: x+r^{m_{1}}, \ldots, x+r^{m_{n}}\in A\right\}\right| \gtrsim_{{n}} \left( \mu_{A}^{n+1}-\delta\right)N,
	\end{equation*}
	where $\mu_{A}:=|A|/N$.
\end{corollary}

Theorems \ref{main theorem} and \ref{popular difference,n-D}  generalize the two-dimensional results of Peluse, Prendiville, and Shao in \cite[Theorems 1.1 and 1.2]{PPS24} to a broader homogeneous setting for the multidimensional polynomial Szemer\'{e}di theorem.  Corollary \ref{popular difference,1-D} is a ``popular'' version of a one-dimensional polynomial Szemer\'{e}di theorem involving $n$ monomials of distinct degrees. It generalizes \cite[Theorem 1.3]{PPS24}. Compared with Shao and Wang’s result \cite[Theorem 1.2]{SW25}, it yields a somewhat better bound in the homogeneous setting.

As in earlier works \cite{PP24, Pel20, PP22, PPS24, SW25}, the overall strategy employs PET induction, concatenation, and a degree‑lowering step to establish an inverse theorem, followed by an energy increment (or density increment) argument that yields the final bound. The core inverse theorem we need follows from a more general result of Kosz, Mirek, Peluse, Wan, and Wright in \cite[Theorem 4.14]{KMPWW24}. Their inverse theorem is established using several technical results on PET induction and concatenation from Kravitz, Kuca, and Leng \cite{KKL24 b}, together with a novel degree‑lowering argument. The main ingredients of this paper are twofold. First, we show how to derive Theorem \ref{popular difference,n-D} from the corresponding inverse theorem; second, we present a modified degree‑lowering argument in Appendix \ref{appendiA}, which differs slightly from that in \cite[Section 4]{KMPWW24}.

Regarding the first ingredient, our derivation follows closely the argument in \cite[Sections 5--7]{PPS24} with the main difference being that we work in the setting of conditional expectation rather than the setting of Fej\'{e}r kernels as in \cite{PPS24}. More specifically, different from \cite[Theorem 4.14]{KMPWW24}, we state the inverse theorem (Theorem \ref{Lambda_{N}-inverse}) in the language of conditional expectation. Our statement is somewhat weaker, but as will be shown in Section \ref{section:strategy}, it is sufficient  for proving Theorem \ref{popular difference,n-D}. The advantage of this adjustment becomes evident when deriving Theorem \ref{popular difference,n-D} from the inverse theorem: it allows us to avoid introducing an extra weight function in the counting operator and allows the energy increment argument to proceed in
a more natural manner, similar to that in \cite{GT09}, without additional computations involving Fej\'{e}r kernels.

Concerning the second ingredient, we make a few further remarks. Peluse \cite{Pel19} introduced the method of degree‑lowering, which has now become a standard tool for obtaining inverse theorems in related settings. Typical degree-lowering arguments in the integer setting  (see for instance \cite[Lemma 6.5]{PP24} and \cite[Lemma 5.5]{PPS24}) aim to lower the Gowers norm degree of structured functions (so‑called ``dual functions''), thereby avoiding the use of the complicated theory of nilsequences.
While of independent interest, this approach has its own complexity in showing the ``low rank'' property of the emerging phase functions, especially in the multidimensional setting. The treatment in our paper aims to minimize this complexity. Concretely, we design a self-improving process in Proposition \ref{degree-lowering} and reduce the inverse theorem (Theorem \ref{Lambda_{N}-inverse with phase}) to the simplest case $n=1$ via simpler ``dual-difference interchange'' results (see Lemma \ref{dual-difference interchange}). Consequently, we only need to handle the complexity of proving the ``low rank'' property of the emerging phase functions in the proof of the simplest case $n=1$ of Theorem \ref{Lambda_{N}-inverse with phase}. Such a modification enables us to avoid using technical lemmas such as \cite[Lemma 4.51]{KMPWW24}.

\textit{This paper is organized as follows.} In Section \ref{section:conventions and general notation}, we set out the conventions and general notation used throughout the paper. In Section \ref{section:strategy} we show how to use the inverse theorem (Theorem \ref{Lambda_{N}-inverse}) to derive Theorem \ref{popular difference,n-D}. In Appendix \ref{appendiA}, we present our modified degree-lowering argument, which, combined with standard PET arguments and concatenation results, actually serves as a complete proof of Theorem \ref{Lambda_{N}-inverse}. In Appendix \ref{appendiB}, we describe a possible way to extend our results to general multidimensional polynomial progressions.


\section{Notation}\label{section:conventions and general notation}

The symbols $\mathbb{N}$, $\mathbb{Z}$, $\mathbb{C}$, and $\mathbb{T}$ denote the sets of positive integers, integers, complex numbers, and the 1‑torus $\mathbb{R}/\mathbb{Z}$, respectively. Unless otherwise stated, non‑bold letters denote elements of \(\mathbb{Z}\), and bold letters denote elements of \(\mathbb{Z}^n\) for some appropriate \(n\), with the exception that the letters \(c, C, \delta\) are reserved for positive real numbers. Variables are usually described only by their ranges. For example:  \(1 \leq H \leq 8.5\) means \(H\) is a positive integer not exceeding \(8.5\), and  \(\delta \in (0, 1/10)\) means \(\delta\) is a real number in the interval \((0, 1/10)\). For \(N \geq 1\) we employ the abbreviations \([N] := \{1, 2, \dots, N\}\) and  \([\pm N] := \{-N, -N+1, \dots, N-1, N\}\).

Under a summation sign \(\Sigma\), a variable written without an explicit range is summed over \(\mathbb{Z}\) or, for a vector‑valued variable, over \(\mathbb{Z}^n\)  for some appropriate \(n\). The symbol \(\mathbf{E}\) denotes either the ordinary expectation operator or the conditional expectation operator (see Definition \ref{conditional expectation}). In the former case, the variables to be averaged and their ranges are indicated below \(\mathbf{E}\).

Unimportant constants are either hidden using the standard asymptotic notation $O(\cdot)$, $\lesssim$, $\gtrsim$, $\asymp$, or simply written as \(c\) (for a ``small'' constant) or \(C\) (for a ``large'' constant), with subscripts indicating dependencies. Within proofs, different occurrences of \(C\) (or \(c\)) may denote different numbers, even on the same line.
In the statements of theorems, lemmas, and propositions, every occurrence of $C$ (or $c$) should be understood as the same constant (we may artificially adjust their values to be identical for the sake of conciseness).

The \(i\)-th component of a vector $\boldsymbol{x}$ is usually written \(x_i\), and the vector obtained from \(\boldsymbol{x}\) by deleting the \(i\)-th component is written \(\widehat{\boldsymbol{x}}^{\,i}\). The symbol \(\boldsymbol{e_j}\) is reserved for the vector whose \(j\)-th component is \(1\) and all other components are \(0\).

For finitely supported functions \(f, g: \mathbb{Z} \to \mathbb{C}\), the \(\ell^k\)-norm (\(k \geq 1\)) is written \(\|f\|_k\) and the \(\ell^2\)-inner product is written \(\langle f, g\rangle\). For a function \(f: \mathbb{Z}^n \to \mathbb{C}\) and a vector \(\boldsymbol{x} \in \mathbb{Z}^n\),  the slice of $f$ corresponding to $\widehat{\boldsymbol{x}}^{i}$ is defined as \(f_{\widehat{\boldsymbol{x}}^{\,i}}(x_i) := f(\boldsymbol{x})\). We say \(f\) is 1‑bounded if \(|f(\boldsymbol{x})| \leq 1\) for every \(\boldsymbol{x} \in \mathbb{Z}^n\).

For $\alpha\in\mathbb{T}$, $\|\alpha\|$ denotes the standard distance to the nearest integer.  The notation $e(x)$ denotes $\exp(2\pi i x)$, and $\mathbf{I}(F)$ equals $1$ if the statement $F$ holds and $0$ otherwise.


\section{Proof of Theorem \ref{popular difference,n-D}}\label{section:strategy}

We use two types of counting operators closely related to Theorem \ref{popular difference,n-D}. The simpler one is
\begin{equation*}
	\Lambda_{N;\boldsymbol{m}}\left(f_{0}, \ldots, f_{n}\right):=\underset{\substack{\boldsymbol{x} \in \prod_{1 \leq i \leq n}\left[N^{m_{i}}\right],\\ r \in [N]}}{\mathbf{E}} f_{0}(\boldsymbol{x}) \prod_{1 \leq j \leq n} f_{j}\left(\boldsymbol{x}+r^{m_{j}} \boldsymbol{e}_{\boldsymbol{j}}\right),
\end{equation*}
where $n\geq 2$, $\boldsymbol{m}=(m_1, \ldots, m_n)\in\mathbb{N}^{n}$, and  $f_{0},\ldots,f_{n}:\mathbb{Z}^{n}\rightarrow\mathbb{C}$ are 1-bounded. A more general version is
\begin{equation*}
	\Lambda_{q, M, \boldsymbol{N};\boldsymbol{m}}\left(f_{0}, \ldots, f_{n}\right):=\underset{\substack{\boldsymbol{x} \in \prod_{1 \leq i \leq n}\left[N_{i}\right],\\r \in [M]}}{\mathbf{E}} f_{0}(\boldsymbol{x}) \prod_{1 \leq j \leq n} f_{j}\left(\boldsymbol{x}+q^{m_{j}} r^{m_{j}} \boldsymbol{e}_{\boldsymbol{j}}\right),
\end{equation*}
where $n\geq 2$, $\boldsymbol{m}\in\mathbb{N}^{n}$, $\boldsymbol{N}=(N_1, \ldots, N_n)\in\mathbb{N}^{n}$, $q,M\geq 1$ satisfying
\begin{equation}\label{range condition}
	qM\leq N_{n}^{1/m_{n}}\leq N_{n-1}^{1/m_{n-1}}\leq\cdots\leq N_{1}^{1/m_{1}},
\end{equation}
and  $f_{0},\ldots,f_{n}:\mathbb{Z}^{n}\rightarrow\mathbb{C}$ are 1-bounded.

To state our inverse theorem for $\Lambda_{N;\boldsymbol{m}}$, let us first recall the definition of conditional expectation in the context of $\mathbb{Z}$.

\begin{defn} \label{conditional expectation}
	Let $f:\mathbb{Z}\rightarrow\mathbb{C}$ and let $\mathcal{B}$ be a partition of $\mathbb{Z}$ with all atoms finite.
	The conditional expectation $\mathbf{E}(f \rvert \mathcal{B})$ is defined as the function which, at each point $x\in\mathbb{Z}$, takes the value equal to the average of $f$ over $\mathcal{B}(x)$, where $\mathcal{B}(x)$ is the atom containing $x$.
			
	Let $\mathcal{B}$ and $\mathcal{B}^{\prime}$ be two partitions of $\mathbb{Z}$. The notation   $\mathcal{B}\vee\mathcal{B}^{\prime}$ denotes the common refinement of $\mathcal{B}$ and $\mathcal{B}^{\prime}$, and $\mathcal{B}\subset\mathcal{B}^{\prime}$ means that  $\mathcal{B}^{\prime}$ refines $\mathcal{B}$.
\end{defn}

In this paper, we deal only with partitions of $\mathbb{Z}$ in which each atom is an arithmetic progression, and all progressions within a single partition share the same common difference and the same finite length.

\begin{defn}
	Let $q,L\geq 1$. The notation $\mathcal{B}_{(qL,q)}$ denotes the following partition of  $\mathbb{Z}$:
	\begin{equation*}
		\mathbb{Z}=\bigsqcup_{s \in \mathbb{Z}, 0 < r \leq q}\{qLs+qk+r: 0 \leq k < L\},
	\end{equation*}
	where the union is disjoint.
	
	More intuitively, $\mathcal{B}_{(qL,q)}=\mathcal{B}_{1}\vee\mathcal{B}_{2}$, where $\mathcal{B}_{1}$ is the partition of $\mathbb{Z}$ into intervals of length $qL$, that is,  $\mathbb{Z}=\bigsqcup_{s \in \mathbb{Z}}(q L s, q L(s+1)]$, and $\mathcal{B}_{2}$ is the partition of $\mathbb{Z}$ into equivalence classes modulo $q$.
\end{defn}

The following is our inverse theorem for the operator $\Lambda_{N;\boldsymbol{m}}$. It can be viewed as a special case of \cite[Theorem 4.14]{KMPWW24}.
\begin{theorem}\label{Lambda_{N}-inverse}
	Given $n\geq 2$ and a strictly increasing $n$-tuple $\boldsymbol{m}$  of positive integers, there exists a constant $C=C\left(n,\boldsymbol{m}\right)>0$ such that the following holds. Let $\delta\in(0,1/10)$, $N\geq 1$, and    $f_{0},\ldots,f_{n}:\mathbb{Z}^{n}\rightarrow\mathbb{C}$ be 1-bounded. If $N \geq \delta^{-C}$ and
	\begin{equation*}
		\left|\Lambda_{N;\boldsymbol{m}}\left(f_{0},\ldots, f_{n}\right)\right| \geq \delta,
	\end{equation*}
	then there exists $1\leq q \leq \delta^{-C}$ such that for any $1 \leq i \leq n$ and any $1\leq L \leq \delta^{C} N$,
	\begin{equation*}
		\underset{\widehat{\boldsymbol{x}}^{i} \in \prod_{1 \leq j \leq n, j \neq i}\left[ N^{m_{j}}\right]}{\mathbf{E}}\left\|\mathbf{E}\left(\left.\left(f_{i}\right)_{\widehat{\boldsymbol{x}}^{i}}\chi_{\left[2N^{m_{i}}\right]} \right\rvert \mathcal{B}_{\left(q^{m_{i}}  L^{m_{i}},q^{m_{i}}\right)}\right)\right\|_{2}^{2} \geq \delta^{C} N^{m_{i}}.
	\end{equation*}
\end{theorem}

In this section, we follow closely the argument in \cite[Sections 5--7]{PPS24} to show how to derive Theorem \ref{popular difference,n-D} from Theorem \ref{Lambda_{N}-inverse}.
The strategy is as follows. We first use Theorem \ref{Lambda_{N}-inverse} to prove the inverse theorem  for the more general operator $\Lambda_{q, M, \boldsymbol{N};\boldsymbol{m}}$ (Theorem \ref{Lambda_{q,M}-inverse}) via an averaging process (thanks to the homogeneity of monomials). We then carry out the energy increment argument to approximate $\Lambda_{q, M, \boldsymbol{N};\boldsymbol{m}}\left(f_{0},\ldots,f_{n}\right)$ by a ``linear object'' (see Proposition \ref{energy-increment}). We finally establish a lower bound for this ``linear object'' (see Proposition \ref{linearization}) and complete the proof of Theorem \ref{popular difference,n-D}.

Let us recall some standard results on conditional expectation (see for instance \cite[Lemma 4.3]{PP22}).

\begin{lemma}\label{property of conditional expectation}
	Let $\mathcal{B},\mathcal{B}^{\prime}$ be partitions  of $\mathbb{Z}$ with all atoms finite, and  $f,g:\mathbb{Z}\rightarrow\mathbb{C}$ be  finitely supported.
	
	\begin{enumerate}[(i)]
		\item\label{property of conditional expectation,(i)} The conditional expectation operator is linear and self-adjoint in the sense that
		\begin{equation*}
			\langle\mathbf{E}(f \rvert \mathcal{B}), g\rangle=\langle f, \mathbf{E}(g \rvert \mathcal{B})\rangle.
		\end{equation*}
		
		\item\label{property of conditional expectation,(ii)} If $\mathcal{B}^{\prime}$ refines $\mathcal{B}$, then
		\begin{equation*}
			\left\|\mathbf{E}\left(f \rvert \mathcal{B}^{\prime}\right)-\mathbf{E}(f \rvert \mathcal{B})\right\|_{2}^{2}=\left\|\mathbf{E}\left(f \rvert \mathcal{B}^{\prime}\right)\right\|_{2}^{2}-\|\mathbf{E}(f \rvert \mathcal{B})\|_{2}^{2}.
		\end{equation*}
		Especially, $\left\|\mathbf{E}\left(f \rvert \mathcal{B}^{\prime}\right)\right\|_{2}\geq \|\mathbf{E}(f \rvert \mathcal{B})\|_{2}$.
		
		\item\label{property of conditional expectation,(iii)} If $\mathcal{B}^{\prime}$ refines $\mathcal{B}$, then
		\begin{equation*}
			\mathbf{E}\left(\mathbf{E}\left(f \rvert \mathcal{B}\right)\rvert\mathcal{B}^{\prime}\right)=\mathbf{E}\left(f \rvert \mathcal{B}\right).
		\end{equation*}
		If $\mathcal{B}$ refines $\mathcal{B}^{\prime}$, then
		\begin{equation*}
			\mathbf{E}\left(\mathbf{E}\left(f \rvert \mathcal{B}\right)\rvert\mathcal{B}^{\prime}\right)=\mathbf{E}\left(f \rvert \mathcal{B}^{\prime}\right).
		\end{equation*}
	\end{enumerate}
\end{lemma}
		
We list some simple properties of $\mathcal{B}_{(qL,q)}$ for later use.
\begin{lemma}\label{property of simple partition}
	Let $k,q,\widetilde{q},L,L_{1},L_{2}\geq 1$, and $f:\mathbb{Z}\rightarrow\mathbb{C}$ be 1-bounded with support in an interval of length $N$.
	\begin{enumerate}[(i)]
		\item\label{property of simple partition,(i)}
		(Formula for the $\ell^{k}$-norm)
		\begin{equation*}
			\left\|\mathbf{E}\left(f \rvert \mathcal{B}_{\left(qL,q\right)}\right)\right\|_{k}^{k}=L \sum_{\substack{s \in \mathbb{Z}, \\ x \in(q L s, q L s+q]}}\left|\underset{0\leq l < L}{\mathbf{E}}f(x+q l)\right|^{k}.
		\end{equation*}
		\item\label{property of simple partition,(ii)}
		(Periodicity)
		\begin{equation*}
			\left\|\mathbf{E}\left(f(\cdot+q L)\rvert \mathcal{B}_{\left(qL,q\right)}\right)\right\|_{k}=\left\|\mathbf{E}\left(f \rvert \mathcal{B}_{\left(qL,q\right)}\right)\right\|_{k}.
		\end{equation*}
		\item\label{property of simple partition,(iii)}
		(``Almost periodicity'')
		\begin{equation*}
			|h|<q \Rightarrow\left\|\mathbf{E}\left(f(\cdot+h)\rvert \mathcal{B}_{\left(qL,q\right)}\right)\right\|_{2}^{2}=\left\|\mathbf{E}\left(f \rvert \mathcal{B}_{\left(qL,q\right)}\right)\right\|_{2}^{2}+O\left(\frac{h}{q} N\right);
		\end{equation*}
		\begin{equation*}
			h=s q, |s|<L \Rightarrow\left\|\mathbf{E}\left(f(\cdot+h) \rvert \mathcal{B}_{\left(qL,q\right)}\right)\right\|_{2}^{2}=\left\|\mathbf{E}\left(f \rvert \mathcal{B}_{\left(qL,q\right)}\right)\right\|_{2}^{2}+O\left(\frac{s}{L} N\right).
		\end{equation*}
		\item\label{property of simple partition,(iv)}
		(``Almost refinement'') Denote $\mathcal{B}_{1}:=\mathcal{B}_{\left(q L_{1}, q\right)}$ and $\mathcal{B}_{2}:=\mathcal{B}_{\left(q \tilde{q} L_{2}, q \tilde{q}\right)}$. If $\tilde{q} L_{2}\leq L_{1}$, then we have
		\begin{equation*}
			\left\|\mathbf{E}\left(f \rvert \mathcal{B}_{1} \vee \mathcal{B}_{2}\right)\right\|_{2}^{2}=\left\|\mathbf{E}\left(f \rvert \mathcal{B}_{2}\right)\right\|_{2}^{2}+O\left(\frac{\tilde{q} L_{2}}{L_{1}} N\right).
		\end{equation*}
	\end{enumerate}	
\end{lemma}
\begin{proof}
	Observe that \eqref{property of simple partition,(i)} holds by definition, while \eqref{property of simple partition,(ii)} and \eqref{property of simple partition,(iii)} follow easily from \eqref{property of simple partition,(i)} by changing variables.
	\eqref{property of simple partition,(iv)} holds because $\mathbf{E}\left(f \rvert \mathcal{B}_{1} \vee \mathcal{B}_{2}\right)$ and $\mathbf{E}\left(f \rvert \mathcal{B}_{2}\right)$ differ on a set of size $O\left( \left(\tilde{q} L_2 / L_1\right) N \right)$.
\end{proof}
\begin{remark}
    All implicit constants in the lemma above can be explicitly chosen as $8$ (for instance).
\end{remark}

The following is our inverse theorem  for the general  operator $\Lambda_{q, M, \boldsymbol{N};\boldsymbol{m}}$.

\begin{theorem}\label{Lambda_{q,M}-inverse}
	Given $n\geq 2$ and a strictly increasing $n$-tuple  $\boldsymbol{m}$ of positive integers, there exists a constant $C=C\left(n,\boldsymbol{m}\right)>0$ such that the following holds. Let $\delta\in(0,1/10)$, $q,M\geq 1$ and $\boldsymbol{N}\in\mathbb{N}^{n}$ satisfy \eqref{range condition}, and   $f_{0},\ldots,f_{n}:\mathbb{Z}^{n}\rightarrow\mathbb{C}$ be 1-bounded. If $M \geq \delta^{-C}$ and
	\begin{equation*}
		\left|\Lambda_{q,M,\boldsymbol{N};\boldsymbol{m}}\left(f_{0},\ldots, f_{n}\right)\right| \geq \delta,
	\end{equation*}
	then there exists $1\leq \tilde{q} \leq \delta^{-C}$ such that for any $1 \leq i \leq n$ and any $1\leq L \leq \delta^{C} M$,
	\begin{equation*}
		\underset{\widehat{\boldsymbol{x}}^{i} \in \prod_{1 \leq j \leq n, j \neq i}\left[ N_{j}\right]}{\mathbf{E}}\left\|\mathbf{E}\left(\left.\left(f_{i}\right)_{\widehat{\boldsymbol{x}}^{i}}\chi_{\left[2N_{i}\right]} \right\rvert \mathcal{B}_{\left(q^{m_{i}}\tilde{q}^{m_{i}}  L^{m_{i}},q^{m_{i}}\tilde{q}^{m_{i}}\right)}\right)\right\|_{2}^{2}\geq \delta^{C} N_{i}.
	\end{equation*}
\end{theorem}

\begin{proof}
    The constants $C$ below depend only on $n$ and $\boldsymbol{m}$, and are not necessarily the same at each occurrence. Without loss of generality, we may assume that the support of $f_{i}$ lies in $\prod_{1 \leq j \leq n}[2^{\mathbf{I}(j=i)} N_{j}]$ for each $0\leq i\leq n$.

	We begin by expressing  $\Lambda_{q,M,\boldsymbol{N};\boldsymbol{m}}$ in the form of an average of $\Lambda_{M;\boldsymbol{m}}$:
	\begin{align*}
		&\Lambda_{q, M, \boldsymbol{N};\boldsymbol{m}}\left(f_{0}, \ldots, f_{n}\right) \\
		=&C_n \underset{\substack{\boldsymbol{x} \in \prod_{1 \leq j \leq n}\left[ \pm 2 N_{j}\right],r \in [M], \\
				\boldsymbol{x}^{\prime} \in \prod_{1 \leq j \leq n}\left[M^{m_{j}}\right]}}{\mathbf{E}} f_{0}\Bigg(\boldsymbol{x}+\sum_{1 \leq j \leq n} q^{m_{j}} x_{j}^{\prime} \boldsymbol{e}_{\boldsymbol{j}}\Bigg)\cdot\\
		&\quad\quad\quad\quad\quad\quad\prod_{1 \leq i \leq n} f_{i}\Bigg(\boldsymbol{x}+\sum_{1 \leq j \leq n, j \neq i} q^{m_{j}} x_{j}^{\prime} \boldsymbol{e}_{\boldsymbol{j}}+q^{m_{i}}\left(x_{i}^{\prime}+r^{m_{i}}\right) \boldsymbol{e}_{\boldsymbol{i}}\Bigg) \\
		=&C_n \underset{\substack{\boldsymbol{x} \in \prod_{1 \leq j \leq n}\left[ \pm 2 N_{j}\right],r \in [M], \\ \boldsymbol{x}^{\prime} \in \prod_{1 \leq j \leq n}\left[M^{m_{j}}\right]}}{\mathbf{E}} f_{0}^{\boldsymbol{x}, q}\left(\boldsymbol{x}^{\prime}\right) \prod_{1 \leq i \leq n} f_{i}^{\boldsymbol{x}, q}\left(\boldsymbol{x}^{\prime}+r^{m_{i}} \boldsymbol{e_{i}}\right) \\
		=& C_n \underset{\boldsymbol{x} \in \prod_{1 \leq j \leq n}\left[ \pm 2 N_{j}\right]}{\mathbf{E}} \Lambda_{M;\boldsymbol{m}}\left(f_{0}^{\boldsymbol{x}, q}, \ldots, f_{n}^{\boldsymbol{x}, q}\right),
	\end{align*}
	where $f_{i}^{\boldsymbol{x}, q}\left(\boldsymbol{x}^{\prime}\right):=f_{i}\left(\boldsymbol{x}+\sum_{1 \leq j \leq n} q^{m_{j}} x_{j}^{\prime} \boldsymbol{e}_{\boldsymbol{j}}\right)$, $\boldsymbol{x}^{\prime}\in \mathbb{Z}^n$.

	Next, we use mainly Theorem \ref{Lambda_{N}-inverse} to arrive at a result (that is, \eqref{Lambda_{q,M}-inverse,eq1} below) close to the desired one, differing by a translation. 	
	By the popularity principle and Theorem \ref{Lambda_{N}-inverse}, for any $\boldsymbol{x}$ in some set $E\subset\prod_{1 \leq j \leq n}[ \pm 2 N_{j}]$ of density $\gtrsim_{n} \delta$, there exists a positive constant $C=C(n,\boldsymbol{m})$ and an integer $\tilde{q}$ (depending on $\boldsymbol{x}$) with $1\leq \tilde{q}\leq \delta^{-C}$   such that for any $1\leq i\leq n$ and any $1\leq L\leq \delta^{C} M$,
	\begin{equation*}
		\underset{\widehat{\boldsymbol{x^{\prime}}}^{i} \in \prod_{1 \leq j \leq n, j \neq i}\left[ M^{m_{j}}\right]}{\mathbf{E}}\left\|\mathbf{E}\left(\left.\left(f_{i}^{\boldsymbol{x}, q}\right)_{\widehat{\boldsymbol{x^{\prime}}}^{i}} \chi_{\left[2 M^{m_{i}}\right]} \right\rvert \mathcal{B}_{\left(\tilde{q}^{m_{i}} L^{m_{i}}, \tilde{q}^{m_{i}}\right)}\right)\right\|_{2}^{2}\geq \delta^{C} M^{m_{i}}.
	\end{equation*}
	Since there are at most $\delta^{-C}$ possible values for $\tilde{q}$, there exists $1\leq \tilde{q}\leq \delta^{-C}$ independent of $\boldsymbol{x}$ such that
	\begin{equation*}
		\underset{\substack{\boldsymbol{x} \in \prod_{1 \leq j \leq n}\left[ \pm 2 N_{j}\right],\\ \widehat{\boldsymbol{x^{\prime}}}^{i} \in \prod_{1 \leq j \leq n, j \neq i}\left[ M^{m_{j}}\right]}}{\mathbf{E}}\left\|\mathbf{E}\left(\left.\left(f_{i}^{\boldsymbol{x}, q}\right)_{\widehat{\boldsymbol{x^{\prime}}}^{i}} \chi_{\left[2 M^{m_{i}}\right]} \right\rvert \mathcal{B}_{\left(\tilde{q}^{m_{i}} L^{m_{i}}, \tilde{q}^{m_{i}}\right)}\right)\right\|_{2}^{2} \geq \delta^{C} M^{m_{i}}.
	\end{equation*}
Inserting the definition of $f_{i}^{\boldsymbol{x}, q}$ and changing variables yields
	\begin{equation*}
		\begin{split}
			&\underset{\boldsymbol{\widehat{x}}^{i} \in \prod_{1 \leq j \leq n,j\ne i}\left[  N_{j}\right]}{\mathbf{E}}\sum_{x_{i}}
			\left\|\mathbf{E}\left(\left. \left(f_{i}\right)_{\boldsymbol{\widehat{x}}^{i}}\left(x_{i}+q^{m_{i}}(\cdot)\right) \chi_{\left[2 M^{m_{i}}\right]} \right\rvert \mathcal{B}_{\left(\tilde{q}^{m_{i}} L^{m_{i}}, \tilde{q}^{m_{i}}\right)}\right)\right\|_{2}^{2} \\
			&\quad\quad\quad \geq \delta^{C} M^{m_{i}}N_{j}.
		\end{split}
	\end{equation*}
Expanding the $\ell^{2}$-norm via Lemma \ref{property of simple partition}\eqref{property of simple partition,(i)} and changing variables yields
	\begin{align*}
	&\underset{\boldsymbol{\widehat{x}}^{i}}{\mathbf{E}}\enspace L^{m_{i}}\sum_{\substack{s^{\prime} \in \mathbb{Z} \\ x_{i}^{\prime} \in\left(\tilde{q}^{m_{i}} L^{m_{i}}s^{\prime}, \tilde{q}^{m_{i}}L^{m_{i}}s^{\prime} +\tilde{q}^{m_{i}}\right]}} \sum_{x_{i}}\\
	&\quad \quad\left|\underset{0 \leq l<L^{m_{i}}}{\mathbf{E}}\left(f_{i}\right)_{\widehat{\boldsymbol{x}}^{i}}\left(x_{i}+q^{m_{i}} \tilde{q}^{m_{i}} l\right) \chi_{\left[2 M^{m_{i}}\right]}\left(x_{i}^{\prime}+\tilde{q}^{m_{i}} l\right)\right|^{2}\geq \delta^{C} M^{m_{i}}N_{j}
	\end{align*}
	(here and below,  the range $\prod_{1 \leq j \leq n, j \neq i}[  N_{j}]$ for  $\widehat{\boldsymbol{x}}^{i}$ is omitted under the expectation operator for simplicity). Note that for any $0 \leq l < L^{m_i}$ and any $x_{i}^{\prime} \in\left(\tilde{q}^{m_{i}} L^{m_{i}}s^{\prime}, \tilde{q}^{m_{i}}L^{m_{i}}s^{\prime} +\tilde{q}^{m_{i}}\right]$,
	\begin{equation*}
		\chi_{[2 M^{m_i}]}(x_i' + \tilde{q}^{m_i} l) =
		\begin{cases}
			1, & \text{if } 0 \leq s' < \left\lfloor \dfrac{2 M^{m_i}}{\tilde{q}^{m_i} L^{m_i}} \right\rfloor, \\[8pt]
			0, & \text{if } s' < 0 \text{ or } s' \geq \left\lfloor \dfrac{2 M^{m_i}}{\tilde{q}^{m_i} L^{m_i}} \right\rfloor + 1.
		\end{cases}
	\end{equation*}
	We can further ignore the possible bad point for $s^{\prime}$ (since, under its size constraint,  $L$ is relatively small) to obtain
	\begin{equation*}
		\underset{\boldsymbol{\widehat{x}}^{i}}{\mathbf{E}}\enspace \sum_{x_{i}}\left|\underset{0 \leq l<L^{m_{i}}}{\mathbf{E}}\left(f_{i}\right)_{\widehat{\boldsymbol{x}}^{i}}\left(x_{i}+q^{m_{i}} \tilde{q}^{m_{i}} l\right) \right|^{2}\geq \delta^{C} N_{i}.
	\end{equation*}
	Furthermore,
	\begin{align*}
			&\underset{\boldsymbol{\widehat{x}}^{i}}{\mathbf{E}}\!\!\!\sum_{\substack{s \in \mathbb{Z}, 0 \leq k<L^{m_{i}}, \\
			x_{i} \in\left(q^{m_{i}} \tilde{q}^{m_{i}} L^{m_{i}} s, q^{m_{i}} \tilde{q}^{m_{i}} L^{m_{i}} s+q^{m_{i}} \tilde{q}^{m_{i}}\right]}}   \!\!\!\!\! \left|\underset{0 \leq l<L^{m_{i}}}{\mathbf{E}}\left(f_{i}\right)_{\widehat{\boldsymbol{x}}^{i}}\left(x_{i}+q^{m_{i}} \tilde{q}^{m_{i}} k+q^{m_{i}} \tilde{q}^{m_{i}} l\right)\right|^{2}\\
			&\,\geq \delta^{C} N_{i}.
	\end{align*}
	By the pigeonhole principle and Lemma \ref{property of simple partition}\eqref{property of simple partition,(i)}, for some $0\leq k<L^{m_{i}}$,
	\begin{equation}\label{Lambda_{q,M}-inverse,eq1}
		\underset{\boldsymbol{\widehat{x}}^{i}}{\mathbf{E}}\!\left\|\mathbf{E}\left(\left.\left(f_{i}\right)_{\widehat{\boldsymbol{x}}^{i}}\left(\cdot+q^{m_{i}} \tilde{q}^{m_{i}} k\right) \right\rvert \mathcal{B}_{\left(q^{m_{i}} \tilde{q}^{m_{i}} L^{m_{i}}, q^{m_{i}} \tilde{q}^{m_{i}}\right)}\right)\right\|_{2}^{2}\geq \delta^{C} N_{i}.
	\end{equation}
	
	Finally, we ``remove'' the translation $q^{m_{i}} \tilde{q}^{m_{i}} k$ above to complete the proof. Let $L$ take the form $K\tilde{L}$ for some $K\geq 1$ to be specified below. Write $k$ as
	\begin{equation*}
		k=a_{0}+b_{0} \tilde{L}^{m_{i}}, \quad 0 \leq a_{0}<\tilde{L}^{m_{i}}, 0 \leq b_{0}<K^{m_{i}}.
	\end{equation*}
	By Lemma \ref{property of simple partition}\eqref{property of simple partition,(iii)}, if $K$ is large (for instance, if $K= \lfloor \delta^{-C}\rfloor$), then we can eliminate the translation $q^{m_{i}}\tilde{q}^{m_{i}}a_{0}$ to obtain
	\begin{equation*}
		\underset{\boldsymbol{\widehat{x}}^{i}}{\mathbf{E}}\enspace\left\|\mathbf{E}\left(\left.\left(f_{i}\right)_{\widehat{\boldsymbol{x}}^{i}}\left(\cdot+b_{0}q^{m_{i}} \tilde{q}^{m_{i}} \tilde{L}^{m_{i}}\right) \right\rvert \mathcal{B}_{\left(q^{m_{i}} \tilde{q}^{m_{i}} L^{m_{i}}, q^{m_{i}} \tilde{q}^{m_{i}}\right)}\right)\right\|_{2}^{2}
		\geq \delta^{C} N_{i}.
	\end{equation*}
	By refining $\mathcal{B}_{(q^{m_{i}} \tilde{q}^{m_{i}} L^{m_{i}}, q^{m_{i}} \tilde{q}^{m_{i}})}$ into $\mathcal{B}_{(q^{m_{i}} \tilde{q}^{m_{i}} \tilde{L}^{m_{i}}, q^{m_{i}} \tilde{q}^{m_{i}})}$ and using Lemmas \ref{property of conditional expectation}\eqref{property of conditional expectation,(ii)} and \ref{property of simple partition}\eqref{property of simple partition,(ii)}, we then obtain the desired result.
\end{proof}

Based on Theorem \ref{Lambda_{q,M}-inverse}, we can now run the energy increment process to give
$\Lambda_{q,M,\boldsymbol{N};\boldsymbol{m}}\left(f_{0},\ldots, f_{n}\right)$ a ``linear approximation''.

\begin{proposition}\label{energy-increment}
	Given $n\geq 2$ and a strictly increasing $n$-tuple $\boldsymbol{m}$ of positive integers, there exists a constant $C=C\left(n,\boldsymbol{m}\right)>0$ such that the following holds. Let $\delta\in(0,1/10)$, $\boldsymbol{N}\in\mathbb{N}^{n}$ satisfy \eqref{s1-2}, and  $f_{0},\ldots,f_{n}:\mathbb{Z}^{n}\rightarrow\mathbb{C}$ be 1-bounded with support in $\prod_{1\leq j\leq n}\left[N_{j}\right]$. If $N_{n}\geq \exp (\delta^{-C})$, then there exist $q,L\geq 1$ such that $(\delta / 8n)L\geq 1$ and
	\begin{equation*}
		\left|\Lambda_{q,\lfloor \frac{\delta}{8 n} L\rfloor,\boldsymbol{N};\boldsymbol{m}}\left(f_{0}, f_{1}, \ldots, f_{n}\right)-\Lambda_{q,\lfloor \frac{\delta}{8 n} L\rfloor,\boldsymbol{N}; \boldsymbol{m}}\left(f_{0}, F_{1}, \ldots, F_{n}\right)\right| \leq \delta ,
	\end{equation*}
	where
	\begin{equation*}
		F_{i}(\boldsymbol{x})=\mathbf{E}\left(\left.\left(f_{i}\right)_{\widehat{\boldsymbol{x}}^{i}} \right\rvert \mathcal{B}_{\left(q^{m_{i}} L^{m_{i}}, q^{m_{i}}\right)}\right)\left(x_{i}\right), \quad 1 \leq i \leq n .
	\end{equation*}
\end{proposition}

\begin{proof}
Throughout this proof, every occurrence of $C$ denotes the same constant, whose value is given in Theorem \ref{Lambda_{q,M}-inverse}. We also adopt the convention of omitting the range $\prod_{1 \leq j \leq n, j \neq i}\left[  N_{j}\right]$ for  $\widehat{\boldsymbol{x}}^{i}$ under the expectation operator.

	We start from the trivial partition $\mathcal{B}_{i}^{(0)}:=\mathcal{B}_{\left(N_{i},1\right)}$, $1\leq i\leq n$, and set
	\begin{equation*}
		F_{i}^{(1)}(\boldsymbol{x})=\mathbf{E}\left(\left.\left(f_{i}\right)_{\widehat{\boldsymbol{x}}^{i}} \right\rvert \mathcal{B}_{i}^{(0)}\right)\left(x_{i}\right) \textrm{ and } L_{0}=N_{n}^{1/m_{n}}.
	\end{equation*}
	If
	\begin{equation}\label{energy-increment,eq1}
		\left|\Lambda_{1, \left\lfloor\frac{\delta}{8 n} L_{0}\right\rfloor,\boldsymbol{N}; \boldsymbol{m}}\left(f_{0}, f_{1}, \ldots, f_{n}\right)-\Lambda_{1, \left\lfloor\frac{\delta}{8 n} L_{0}\right\rfloor,\boldsymbol{N};\boldsymbol{m}}\left(f_{0}, F_{1}^{(1)}, \ldots, F_{n}^{(1)}\right)\right| \geq \delta,
	\end{equation}
	we claim that if $L_{0}$ is large, then there exists $1\leq q_{1} \leq \delta^{-nC}$ such that the following inequality
	\begin{equation}\label{s3-1}
		\underset{\widehat{\boldsymbol{x}}^{i}}{\mathbf{E}}\enspace\left\|\mathbf{E}\left(\left.\left(f_{i}\right)_{\widehat{\boldsymbol{x}}^{i}} \right\rvert \mathcal{B}_{i}^{(1)}\right)\right\|_{2}^{2}-\underset{\widehat{\boldsymbol{x}}^{i}}{\mathbf{E}}\enspace\left\|\mathbf{E}\left(\left.\left(f_{i}\right)_{\widehat{\boldsymbol{x}}^{i}} \right\rvert \mathcal{B}_{i}^{(0)}\right)\right\|_{2}^{2}  \geq \frac{1}{2}\delta^{C} N_{i}
	\end{equation}
	holds for at least one index $i\in [n]$, where
	\begin{equation*}
		\mathcal{B}_{i}^{(1)}:=\mathcal{B}_{(q_{1}^{m_{i}}L_{1}^{m_{i}}, q_{1}^{m_{i}})}\text{ with }L_{1}:= \left\lfloor\frac{\delta^{1+(n+1)C}}{2^{4} n} L_{0}\right\rfloor.
	\end{equation*}
		To prove this claim, by a telescoping argument and the triangle inequality, we obtain from \eqref{energy-increment,eq1} that at least one of the $n$ quantities
	\begin{align*}
		&\left|\Lambda_{1, \left\lfloor\frac{\delta}{8 n} L_{0}\right\rfloor,\boldsymbol{N}; \boldsymbol{m}}\left(f_{0}, f_{1}-F_{1}^{(1)}, f_2, f_3, \ldots, f_{n}\right)\right|,\\
		&\left|\Lambda_{1, \left\lfloor\frac{\delta}{8 n} L_{0}\right\rfloor,\boldsymbol{N}; \boldsymbol{m}}\left(f_{0}, F_{1}^{(1)}, f_2-F_{2}^{(1)}, f_3, \ldots, f_{n}\right)\right|,\\
		&\ldots,\\
		&\left|\Lambda_{1, \left\lfloor\frac{\delta}{8 n} L_{0}\right\rfloor,\boldsymbol{N};\boldsymbol{m}}\left(f_{0}, F_{1}^{(1)}, F_{2}^{(1)},\ldots, F_{n-1}^{(1)}, f_{n}-F_{n}^{(1)}\right)\right|
	\end{align*}
	is at least $\delta/n$. Without loss of generality,  we assume that the first $k$ quantities are at least $\delta/n$. By Theorem \ref{Lambda_{q,M}-inverse}, if $L_{0}$ is large, then for any $1\leq i\leq k$, there exists $1\leq p_{i} \leq \delta^{-C}$ such that for any
	\begin{equation*}
			1\leq R_{i}\leq \frac{1}{2^4 n} \delta^{1+C}L_{0},
	\end{equation*}
	we have
	\begin{equation*}
		\underset{\widehat{\boldsymbol{x}}^{i} }{\mathbf{E}}\enspace\left\|\mathbf{E}\left(\left.\left(f_{i}\right)_{\widehat{\boldsymbol{x}}^{i}}-\left(F_{i}^{(1)}\right)_{\widehat{\boldsymbol{x}}^{i}} \right\rvert \mathcal{B}_{\left(p_{i}^{m_{i}} R_{i}^{m_{i}}, p_{i}^{m_{i}}\right)}\right)\right\|_{2}^{2} \geq \delta^{C} N_{i}.
	\end{equation*}
Denote 	
	\begin{equation*}
		q_{1}=\prod_{1 \leq j \leq k} p_{j} \leq\delta^{-nC}
	\end{equation*}
	and require $R_{i}$, $1\leq i\leq k$, to take the form
	\begin{equation*}
		R_{i}=\bigg(\prod_{1 \leq j \leq k, j \neq i} p_{j}\bigg) L_{1} \textrm{ with } L_{1}:=\left\lfloor\frac{\delta^{1+(n+1)C}}{2^{4} n} L_{0}\right\rfloor.
	\end{equation*}
	Then, by Lemma \ref{property of conditional expectation}\eqref{property of conditional expectation,(ii)}, for any $1\leq i\leq k$,
	\begin{equation*}
		\underset{\widehat{\boldsymbol{x}}^{i}}{\mathbf{E}}\enspace\left\|\mathbf{E}\left(\left.\left(f_{i}\right)_{\widehat{\boldsymbol{x}}^{i}}-\left(F_{i}^{(1)}\right)_{\widehat{\boldsymbol{x}}^{i}} \right\rvert \mathcal{B}_{i}^{(1)}\right)\right\|_{2}^{2} \geq \delta^{C} N_{i}
	\end{equation*}
	with $\mathcal{B}_{i}^{(1)}$ defined as above. By Lemma  \ref{property of conditional expectation}, we further obtain
	\begin{equation*}
		\underset{\widehat{\boldsymbol{x}}^{i}}{\mathbf{E}}\enspace\left\|\mathbf{E}\left(\left.\left(f_{i}\right)_{\widehat{\boldsymbol{x}}^{i}} \right\rvert \mathcal{B}_{i}^{(0)}\vee\mathcal{B}_{i}^{(1)}\right)\right\|_{2}^{2}-\underset{\widehat{\boldsymbol{x}}^{i}}{\mathbf{E}}\enspace\left\|\mathbf{E}\left(\left.\left(f_{i}\right)_{\widehat{\boldsymbol{x}}^{i}} \right\rvert \mathcal{B}_{i}^{(0)}\right)\right\|_{2}^{2}
		\geq \delta^{C} N_{i}.
	\end{equation*}
	Since $q_{1}L_{1}\leq 2^{-4}\delta^{C} L_{0}$, we know from Lemma \ref{property of simple partition}\eqref{property of simple partition,(iv)} that
	\begin{equation*}
		\underset{\widehat{\boldsymbol{x}}^{i}}{\mathbf{E}}\enspace\left\|\mathbf{E}\left(\left.\left(f_{i}\right)_{\widehat{\boldsymbol{x}}^{i}} \right\rvert \mathcal{B}_{i}^{(0)}\vee\mathcal{B}_{i}^{(1)}\right)\right\|_{2}^{2}-\underset{\widehat{\boldsymbol{x}}^{i}}{\mathbf{E}}\enspace\left\|\mathbf{E}\left(\left.\left(f_{i}\right)_{\widehat{\boldsymbol{x}}^{i}} \right\rvert \mathcal{B}_{i}^{(1)}\right)\right\|_{2}^{2}
		\leq \frac{1}{2}\delta^{C} N_{i}.
	\end{equation*}
Combining the two inequalities above, we obtain the inequality \eqref{s3-1} for any $1\leq i\leq k$. This proves the claim.
	
We proceed by iterating the above step. To demonstrate precisely how this is done, we show a few further iterations in detail. We set $F_{i}^{(2)}(\boldsymbol{x})=\mathbf{E}(\left(f_{i}\right)_{\widehat{\boldsymbol{x}}^{i}} \rvert \mathcal{B}_{i}^{(1)})\left(x_{i}\right)$, $1\leq i\leq n$. If 	
	\begin{equation*}
		\left|\Lambda_{q_{1},\left\lfloor \frac{\delta}{8 n} L_{1}\right\rfloor,\boldsymbol{N};\boldsymbol{m}}\left(f_{0}, f_{1}, \ldots, f_{n}\right)-\Lambda_{q_{1},\left\lfloor \frac{\delta}{8 n} L_{1}\right\rfloor,\boldsymbol{N}; \boldsymbol{m}}\left(f_{0}, F_{1}^{(2)}, \ldots, F_{n}^{(2)}\right)\right|
		\geq \delta,
	\end{equation*}
	then we argue similarly to obtain (note that the support of $F_{i}^{(2)}$ lies in $\left[2N_{i}\right]$) that  if $L_{1}$ is large, then there exists $1\leq q_{2} \leq \delta^{-nC}$
	such that the following inequality
	\begin{equation*}
		\underset{\widehat{\boldsymbol{x}}^{i}}{\mathbf{E}}\enspace\left\|\mathbf{E}\left(\left.\left(f_{i}\right)_{\widehat{\boldsymbol{x}}^{i}} \right\rvert \mathcal{B}_{i}^{(2)}\right)\right\|_{2}^{2}-\underset{\widehat{\boldsymbol{x}}^{i}}{\mathbf{E}}\enspace\left\|\mathbf{E}\left(\left.\left(f_{i}\right)_{\widehat{\boldsymbol{x}}^{i}} \right\rvert \mathcal{B}_{i}^{(1)}\right)\right\|_{2}^{2}  \geq \frac{1}{2}\delta^{C} N_{i}
	\end{equation*}
 holds for at least one index $i\in [n]$, where
 \begin{equation*}
	\mathcal{B}_{i}^{(2)}:=\mathcal{B}_{\left(q_{1}^{m_{i}} q_{2}^{m_{i}} L_{2}^{m_{i}}, q_{1}^{m_{i}} q_{2}^{m_{i}}\right)} \textrm{ with }	 L_{2}:=\left \lfloor\frac{\delta^{1+(n+1)C}}{2^{4} n}  L_{1}\right\rfloor.
\end{equation*}

	Next, we set $F_{i}^{(3)}(\boldsymbol{x})=\mathbf{E}(\left(f_{i}\right)_{\widehat{\boldsymbol{x}}^{i}} \rvert \mathcal{B}_{i}^{(2)})\left(x_{i}\right)$, $1\leq i\leq n$. If 	
	\begin{align*}
		&\left|\Lambda_{q_{1}q_{2},\left\lfloor \frac{\delta}{8 n} L_{2}\right\rfloor,\boldsymbol{N};\boldsymbol{m}}\left(f_{0}, f_{1}, \ldots, f_{n}\right)-\Lambda_{q_{1}q_{2},\left\lfloor \frac{\delta}{8 n} L_{2}\right\rfloor,\boldsymbol{N}; \boldsymbol{m}}\left(f_{0}, F_{1}^{(3)}, \ldots, F_{n}^{(3)}\right)\right|\\
		&\geq \delta,
	\end{align*}
	then we obtain that if $L_{2}$ is large, then there exists $1\leq q_{3} \leq \delta^{-nC}$
	such that the following inequality
	\begin{equation*}
		\underset{\widehat{\boldsymbol{x}}^{i}}{\mathbf{E}}\enspace\left\|\mathbf{E}\left(\left.\left(f_{i}\right)_{\widehat{\boldsymbol{x}}^{i}} \right\rvert \mathcal{B}_{i}^{(3)}\right)\right\|_{2}^{2}-\underset{\widehat{\boldsymbol{x}}^{i}}{\mathbf{E}}\enspace\left\|\mathbf{E}\left(\left.\left(f_{i}\right)_{\widehat{\boldsymbol{x}}^{i}} \right\rvert \mathcal{B}_{i}^{(2)}\right)\right\|_{2}^{2}  \geq \frac{1}{2}\delta^{C} N_{i}
	\end{equation*}
 holds for at least one index $i\in [n]$, where
		\begin{equation*}
		\mathcal{B}_{i}^{(3)}:=\mathcal{B}_{\left(q_{1}^{m_{i}} q_{2}^{m_{i}}q_{3}^{m_{i}} L_{3}^{m_{i}}, q_{1}^{m_{i}} q_{2}^{m_{i}}q_{3}^{m_{i}}\right)} \text{ with }L_{3}:= \left\lfloor\frac{\delta^{1+(n+1)C}}{2^{4} n}  L_{2}\right\rfloor.
	\end{equation*}
	
	Limited by the natural upper bound of energy,
	the process above can be iterated at most $s:=n \lfloor 2 \delta^{-C}  \rfloor$ times. Thus, at the final stage, we are guaranteed to obtain
	\begin{equation*}
		\begin{split}
			&\left|\Lambda_{q_{1} \cdots q_{s},\left\lfloor\!\frac{\delta}{8 n} L_{s}\!\right\rfloor,\boldsymbol{N}; \boldsymbol{m}}\!\left(\!f_{0}, f_{1}, \ldots, f_{n}\right)\!-\!\Lambda_{q_{1} \cdots q_{s},\left\lfloor\!\frac{\delta}{8 n} L_{s}\!\right\rfloor,\boldsymbol{N}; \boldsymbol{m}}\!\left(\!f_{0}, F_{1}^{(s+1)}\!,\ldots, F_{n}^{(s+1)}\!\right)\!\right| \\
			&\leq \delta,
		\end{split}
	\end{equation*}
	where $1\leq q_{1},\ldots,q_{s}\leq \delta^{-nC}$,
	\begin{equation*}
		L_{i+1}:=\left\lfloor\frac{\delta^{1+(n+1)C}}{2^{4} n} L_{i}\right\rfloor,\quad 1 \leq i \leq s-1,
	\end{equation*}
	and
	\begin{equation*}
		F_{i}^{(s+1)}(\boldsymbol{x})=\mathbf{E}\left(\left.\left(f_{i}\right)_{\widehat{\boldsymbol{x}}^{i}} \right\rvert \mathcal{B}_{\left(q_{1}^{m_{i}} \ldots q_{s}^{m_{i}} L_{s}^{m_{i}}, q_{1}^{m_{i}} \ldots q_{s}^{m_{i}}\right)}\right)\left(x_{i}\right), \quad   1 \leq i \leq s.
	\end{equation*}
For $N_{n}$ sufficiently large, we have $(\delta / 8n)L_{s}\geq 1$. The size requirement for $N_{n}$ to ensure this is derived from a simple computation.
\end{proof}

\begin{remark}
	It follows from the proof above that $q$ and $L$ appearing in the statement of Proposition \ref{energy-increment} satisfy
	\begin{equation*}	
		1\leq q \leq \exp \left(\delta^{-C}\right),\quad L \geq \left\lfloor\exp \left(-\delta^{-C}\right) N_{n}^{1/m_{n}}\right\rfloor.
	\end{equation*}
\end{remark}

In order to establish a lower bound for this ``linear approximation'' in Proposition \ref{energy-increment}, we need an elementary inequality.

\begin{lemma}\label{roy's inequality}
	Let $n\geq 1$, $E_{1},\ldots,E_{n}$ be non-empty finite sets, and  $A=\prod_{1 \leq i \leq n} E_{i}$ be the corresponding product set. Then for any function $f: A \rightarrow[0,1]$, we have
	\begin{equation*}
		\underset{\boldsymbol{x}, \boldsymbol{x}^{\prime} \in A}{\mathbf{E}} f(\boldsymbol{x}) \prod_{1 \leq j \leq n} f\left(\boldsymbol{x}+\left(x_{j}^{\prime}-x_{j}\right) \boldsymbol{e}_{\boldsymbol{j}}\right) \gtrsim_{n}\left(\underset{\boldsymbol{x} \in A}{\mathbf{E}} f(\boldsymbol{x})\right)^{n+1}.
	\end{equation*}		
\end{lemma}

\begin{proof}
For the proof, see \cite[Lemma 3.1]{HLY21} for the case $n=2$, while the general case is established similarly by induction on $n$.
\end{proof}

Lemma \ref{roy's inequality} yields the following corollary, which will be readily employed in our subsequent arguments. See \cite[Lemma 5.3]{SW25} for a similar result in one dimension.

\begin{lemma}\label{roy's inequality,corollary}
	Let $n\geq 1$, $N_{i},q_{i},L_{i}\geq 1$ for all $1\leq i\leq n$, and  $A=\prod_{1\leq j\leq n}\left[N_{j}\right]$. Then for any function $f:A\rightarrow[0,1]$, we have
	\begin{equation}\label{roy's inequality,corollary,eq1}
		\underset{\boldsymbol{x} \in A}{\mathbf{E}} f(\boldsymbol{x})\! \prod_{1 \leq i \leq n} \mathbf{E}\left(\left.f_{\widehat{\boldsymbol{x}}^{i}} \right\rvert \mathcal{B}_{\left(q_{i} L_{i}, q_{i}\right)}\right)\left(x_{i}\right) \gtrsim_{n}\left(\underset{x \in A}{\mathbf{E}} f(\boldsymbol{x})\right)^{n+1}.
	\end{equation}
\end{lemma}

\begin{proof}
	The left-hand side of \eqref{roy's inequality,corollary,eq1} can be expanded fully as
	\begin{equation*}
		\begin{split}
			\frac{\prod_{1 \leq j \leq n} L_{j}}{\prod_{1 \leq j \leq n} N_{j}}\!\!\!\!\! &\sum_{\substack{\boldsymbol{s}\in \mathbb{Z}^{n},\\ \boldsymbol{x}\in \prod_{1 \leq i \leq n}\left(q_{i} L_{i} s_{i}, q_{i} L_{i} s_{i}+q_{i}\right]}} \underset{\boldsymbol{k,l}\in\prod_{1 \leq i \leq n}\left[0,L_{i}\right)}{\mathbf{E}}\\
			&f\Bigg(\boldsymbol{x}+\sum_{1 \leq i\leq n} q_{i} k_{i} \boldsymbol{e_{i}}\Bigg) \prod_{1 \leq i \leq n} f\Bigg(\boldsymbol{x}+\sum_{1 \leq j \leq n, j \neq i} q_{j} k_{j} \boldsymbol{e_{j}}+q_{i} l_{i} \boldsymbol{e_{i}}\Bigg).
		\end{split}
	\end{equation*}
Then, by Lemma \ref{roy's inequality} and Jensen's inequality, it is
	\begin{equation*}
		\begin{split}
			&\gtrsim_{n} \frac{\prod_{1 \leq j \leq n} L_{j}}{\prod_{1 \leq j \leq n} N_{j}}\!\!\!\!\sum_{\substack{\boldsymbol{s}\in \mathbb{Z}^{n},\\ \boldsymbol{x}\in \prod_{1 \leq i \leq n}\left(q_{i} L_{i} s_{i}, q_{i} L_{i} s_{i}+q_{i}\right]}}\!\!\!\!\Biggl(\!\underset{\boldsymbol{k}\in\prod_{1 \leq i \leq n}\left[0,L_{i}\right)}{\mathbf{E}} \!f\Biggl(\!\boldsymbol{x}+\sum_{1 \leq i \leq n} q_{i} k_{i} \boldsymbol{e}_{i}\!\!\Biggr)\!\Biggr)^{\!\! n+1}\\
			&\geq \Biggl(\underset{\boldsymbol{k}\in\prod_{1 \leq i \leq n}\left[0,L_{i}\right)}{\mathbf{E}} \frac{\prod_{1 \leq j \leq n} L_{j}}{\prod_{1 \leq j \leq n} N_{j}}\!\!\!\! \sum_{\substack{\boldsymbol{s}\in \mathbb{Z}^{n},\\ \boldsymbol{x}\in \prod_{1 \leq i \leq n}\left(q_{i} L_{i} s_{i}, q_{i} L_{i} s_{i}+q_{i}\right]}}\!\!\!\!\!\!f\Biggl(\boldsymbol{x}+\sum_{1 \leq i \leq n} q_{i} k_{i} \boldsymbol{e}_{i}\Biggr)\!\Biggr)^{\!\!n+1}\\
			&=\left(\underset{x \in A}{\mathbf{E}} f(\boldsymbol{x})\right)^{n+1},
		\end{split}
	\end{equation*}
	as desired.
\end{proof}

We now provide a lower bound for the ``linear approximation''.

\begin{proposition}\label{linearization}
	Let $n\geq 1$ and $\boldsymbol{m}\in\mathbb{N}^{n}$. Let  $\delta\in(0,1)$, $q,M,L\geq 1$ and $\boldsymbol{N}\in\mathbb{N}^{n}$ such that \eqref{range condition} holds and  $M \leq(\delta / 8n) L$. Then for any function $f:\mathbb{Z}^{n}\rightarrow[0,1]$ with support in $\prod_{1 \leq i \leq n}\left[N_{i}\right]$, we have
	\begin{equation*}
		\Lambda_{q, M, \boldsymbol{N};\boldsymbol{m}}\left(f, F_{1}, \ldots, F_{n}\right) \geq c_{n}\left(\underset{\boldsymbol{x} \in\prod_{1 \leq i \leq n}\left[N_{i}\right]}{\mathbf{E}} f(\boldsymbol{x})\right)^{n+1}-\frac{1}{2}\delta
	\end{equation*}
	for some constant $c_{n}>0$, where
	\begin{equation*}
		F_{i}(\boldsymbol{x}):=\mathbf{E}\left(\left.f_{\widehat{\boldsymbol{x}}^{i}} \right\rvert \mathcal{B}_{\left(q^{m_{i}} L^{m_{i}}, q^{m_{i}}\right)}\right)\left(x_{i}\right), \quad 1 \leq i \leq n.
	\end{equation*}
\end{proposition}

\begin{proof}
	Note that for any fixed $1\leq i\leq n$, $1\leq r\leq M$, $s\in\mathbb{Z}$, and $\widehat{\boldsymbol{x}}^{i} \in \prod_{1 \leq j \leq n, j \neq i}\left[N_{j}\right]$, there are at most $q^{m_{i}}r^{m_{i}}$ ``bad'' integers $x_{i}$ in the interval $\left(q^{m_{i}}L^{m_{i}} s, q^{m_{i}}L^{m_{i}}(s+1)\right]$ such that
	\begin{equation*}
		\mathbf{E}\left(\left.f_{\widehat{\boldsymbol{x}}^{i}} \right\rvert \mathcal{B}_{\left(q^{m_{i}} L^{m_{i}}, q^{m_{i}}\right)}\right)\left(x_{i}+q^{m_{i}} r^{m_{i}}\right)\ne\mathbf{E}\left(\left.f_{\widehat{\boldsymbol{x}}^{i}} \right\rvert \mathcal{B}_{\left(q^{m_{i}} L^{m_{i}}, q^{m_{i}}\right)}\right)\left(x_{i}\right).
	\end{equation*}
	Hence
	\begin{equation*}
		\begin{split}
			&\underset{\substack{\widehat{\boldsymbol{x}}^{i} \in \prod_{1 \leq j \leq n, j \neq i}\left[N_{j}\right],\\ r \in[M]}}{\mathbf{E}} \frac{1}{N_{i}} \sum_{x_{i}}\\
			&\left|\mathbf{E}\left(\left.f_{\widehat{\boldsymbol{x}}^{i}} \right\rvert \mathcal{B}_{\left(q^{m_{i}} L^{m_{i}}, q^{m_{i}}\right)}\right)\left(x_{i}+q^{m_{i}} r^{m_{i}}\right)-\mathbf{E}\left(\left.f_{\widehat{\boldsymbol{x}}^{i}} \right\rvert \mathcal{B}_{\left(q^{m_{i}} L^{m_{i}}, q^{m_{i}}\right)}\right)\left(x_{i}\right)\right|
		\end{split}
	\end{equation*}
	has an upper bound $4 (M/L)^{m_i} \leq \delta / (2n)$. We can thus replace, for all $1\leq i\leq n$, the term
	\begin{equation*}
		\mathbf{E}\left(\left.f_{\widehat{\boldsymbol{x}}^{i}} \right\rvert \mathcal{B}_{\left(q^{m_{i}} L^{m_{i}}, q^{m_{i}}\right)}\right)\left(x_{i}+q^{m_{i}} r^{m_{i}}\right)
	\end{equation*}
by
	\begin{equation*}
	\mathbf{E}\left(\left.f_{\widehat{\boldsymbol{x}}^{i}} \right\rvert \mathcal{B}_{\left(q^{m_{i}} L^{m_{i}}, q^{m_{i}}\right)}\right)\left(x_{i}\right)
	\end{equation*}
	at a total cost not exceeding $\delta/2$ in the expression of $\Lambda_{q, M, \boldsymbol{N};\boldsymbol{m}}(f, F_{1}, \ldots, F_{n})$. The desired result then follows from Lemma \ref{roy's inequality,corollary}.
\end{proof}
	
We are now ready to prove Theorem \ref{popular difference,n-D}.

\begin{proof}[Proof of Theorem \ref{popular difference,n-D}]
	 If $N_{n}\geq \exp (\delta^{-C})$ for some constant $C=C\left(n,\boldsymbol{m}\right)$, by Propositions \ref{energy-increment} and \ref{linearization}, for some $q,L\geq 1$, we have $(\delta / 8n) L\geq 1$ and
	\begin{equation*}
		\Lambda_{q,\left\lfloor \frac{\delta}{8 n} L\right\rfloor,\boldsymbol{N};\boldsymbol{m} }\left(\chi_{A},  \ldots, \chi_{A}\right) \geq c_{n}\mu_{A}^{n+1}-\frac{3}{2}\delta.
	\end{equation*}
	The result then follows by replacing $\delta$ by $\frac{2}{3}c_{n} \delta$ and applying the pigeonhole principle.
\end{proof}

We conclude this section with the proof of Corollary \ref{popular difference,1-D}.
\begin{proof}[Proof of Corollary \ref{popular difference,1-D}]
	Without loss of generality, we assume $N^{1/m_{n}}$ is an integer. Associate $A$ with
	\begin{equation*}
		A^{\prime}:=\Bigg\{\boldsymbol{x} \in \prod_{1 \leq i \leq n}\left[N^{m_{i}/m_{n}}\right]: x_{1}+\cdots+x_{n} \in A\Bigg\}.
	\end{equation*}
	The relationship between the size of $A^{\prime}\subset \prod_{1 \leq i \leq n}[N^{m_{i}/m_{n}}]$ and the size of $A\subset [N]$ is as follows
	\begin{equation*}
		\begin{split}
			\left|A^{\prime}\right| & =\sum_{x_{i} \in\left[N^{m_{i}/m_{n}}\right],1\leq i\leq n-1}\left|\left(A-x_{1}-\cdots-x_{n-1}\right) \cap[N]\right| \\
			& \geq\left(|A|-(n-1) 	N^{\frac{m_{n-1}}{m_{n}}}\right) N^{\frac{m_{1}+\cdots+m_{n-1}}{m_{n}}}.
		\end{split}
	\end{equation*}	
	Hence, $\mu_{A^{\prime}} \geq \mu_{A}-\delta$ since $N$ is sufficiently large.
	
	By Theorem \ref{popular difference,n-D} and the definition of $A^{\prime}$, there exists an integer $r\geq 1$ such that
	\begin{equation*}
		\begin{split}
			&\Bigg|\Bigg\{\boldsymbol{x} \in \prod_{1 \leq i \leq n}\left[N^{\frac{m_{i}}{m_{n}}}\right]: \sum_{1 \leq j \leq n} x_{j}, r^{m_{1}}+\sum_{1 \leq j \leq n} x_{j}, \ldots, r^{m_{n}}+\sum_{1 \leq j \leq n} x_{j}\in A\Bigg\}\Bigg|\\
			&\gtrsim_{n} \left( \mu_{A^{\prime}}^{n+1}- \delta\right)N^{\frac{m_{1}+\cdots+ m_{n}}{m_{n}}}\\
			&\geq \left(\mu_{A}^{n+1}-2^{n+1} \delta\right) N^{\frac{m_{1}+\cdots+m_{n}}{m_{n}}}.
		\end{split}
	\end{equation*}
	The result then follows from the pigeonhole principle for  $x_{1},\ldots, x_{n-1}$.
\end{proof}

\appendix
\section{A modified degree-lowering argument} \label{appendiA}

In this section, in the course of proving Theorem \ref{Lambda_{N}-inverse}, we provide our modified degree‑lowering argument.  In fact, we will prove a stronger inverse theorem (Theorem \ref{Lambda_{N}-inverse with phase}) for the following operator with phase functions
\begin{align*}
	&\Lambda_{N;\boldsymbol{m}}^{\alpha_{1}, \ldots, \alpha_{k}}\left(f_{0}, \ldots, f_{n}\right):= \\
	&\underset{\substack{\boldsymbol{x} \in \prod_{1 \leq i \leq n}\left[N^{m_{i}}\right], \\
			r \in [N]}}{\mathbf{E}} f_{0}(\boldsymbol{x}) \Bigg(\prod_{1 \leq i \leq n} f_{i}\left(\boldsymbol{x}+r^{m_{i}} \boldsymbol{e}_{\boldsymbol{i}}\right)\Bigg)
	e\Bigg(\sum_{1\leq j\leq k}\alpha_{j}(\boldsymbol{x})r^{m_{n+j}}\Bigg),
\end{align*}
where $n\geq 1$, $k\geq 0$, $\boldsymbol{m}\in\mathbb{N}^{n+k}$, $\alpha_{1},\ldots,\alpha_{k}:\mathbb{Z}^{n}\rightarrow\mathbb{T}$ are phase functions, and  $f_{0},\ldots,f_{n}:\mathbb{Z}^{n}\rightarrow\mathbb{C}$ are 1-bounded functions.  When $k=0$, it coincides with the operator $\Lambda_{N;\boldsymbol{m}}$ defined at the beginning of Section \ref{section:strategy}.

\begin{theorem}\label{Lambda_{N}-inverse with phase}
	Given $n\geq 1$, $k\geq 0$, and a strictly increasing $(n+k)$-tuple $\boldsymbol{m}$ of positive integers, there exists a constant $C=C\left(n,k,\boldsymbol{m}\right)>0$ such that the following holds. Let $\delta\in(0,1/10)$, $\alpha_{1},\ldots,\alpha_{k}:\mathbb{Z}^{n}\rightarrow\mathbb{T}$ be phase functions, and   $f_{0},\ldots,f_{n}:\mathbb{Z}^{n}\rightarrow\mathbb{C}$ be 1-bounded functions. If $N \geq \delta^{-C}$ and
	\begin{equation*}
		\left|\Lambda_{N;\boldsymbol{m}}^{\alpha_{1},\ldots,\alpha_{k}}\left(f_{0},\ldots, f_{n}\right)\right| \geq \delta,
	\end{equation*}
	then there exists $1\leq q \leq \delta^{-C}$ such that for any $1 \leq i \leq n$ and any $1\leq L \leq \delta^{C} N$,
	\begin{equation*}
		\underset{\widehat{\boldsymbol{x}}^{i} \in \prod_{1 \leq j \leq n, j \neq i}\left[ N^{m_{j}}\right]}{\mathbf{E}}\left\|\mathbf{E}\left(\left.\left(f_{i}\right)_{\widehat{\boldsymbol{x}}^{i}}\chi_{\left[2N^{m_{i}}\right]} \right\rvert \mathcal{B}_{\left(q^{m_{i}}  L^{m_{i}},q^{m_{i}}\right)}\right)\right\|_{2}^{2} \geq \delta^{C} N^{m_{i}}.
	\end{equation*}
\end{theorem}
\begin{remark}
	For the case $n=1$, the conclusion above should be understood as
	\begin{equation*}
		\left\|\mathbf{E}\left(\left.f_{1}\chi_{\left[2N^{m_{1}}\right]} \right\rvert \mathcal{B}_{\left(q^{m_{1}}  L^{m_{1}},q^{m_{1}}\right)}\right)\right\|_{2}^{2} \geq \delta^{C} N^{m_{1}}.
	\end{equation*}
	
\end{remark}

The main ingredients of this section lie in two aspects. We first show that the simplest case $n=1$ of Theorem \ref{Lambda_{N}-inverse with phase} implies the general $n$ case; we then prove the case $n=1$ itself.

Let us first review some basics of difference operators and Gowers norms.
\begin{defnprop}\label{Gowers norm}
	\leavevmode
	\begin{enumerate}[(i)]
		\item\label{Gowers norm,(i)}
		Let $s\geq 1$. For a function $f:\mathbb{Z}\rightarrow\mathbb{C}$, define its multiplicative difference and additive difference with respect to $h\in\mathbb{Z}$ and $\boldsymbol{h}=(h_1,\ldots, h_s)\in\mathbb{Z}^{s}$ by
		\begin{equation*}
			\Delta_{h} f:=f(\cdot) \overline{f(\cdot+h)},\quad \Delta_{\boldsymbol{h}} f:=\Delta_{h_{1}} \cdots \Delta_{h_{s}} f,
		\end{equation*}
		and
		\begin{equation*}
			\partial_{h} f:=f(\cdot)-f(\cdot+h) ,\quad \partial_{\boldsymbol{h}} f:=\partial_{h_{1}} \cdots \partial_{h_{s}} f.
		\end{equation*}
		
		Note that $\Delta_{\boldsymbol{h}}$ is multiplicative ($\Delta_{\boldsymbol{h}}(fg)=\Delta_{\boldsymbol{h}}f \Delta_{\boldsymbol{h}}g$) and commutative (its value is independent of the order of $h_1, \ldots, h_s$). Moreover, for any $\alpha:\mathbb{Z}\rightarrow\mathbb{T}$, we have the  identity $\Delta_{\boldsymbol{h}} e(\alpha)=e\left(\partial_{\boldsymbol{h}}\alpha\right)$.
		
		\item\label{Gowers norm,(ii)} Let $s\geq 1$. The Gowers norm of order $s$ for a finitely supported function $f: \mathbb{Z} \rightarrow \mathbb{C}$ is defined by
		\begin{equation*}
			\|f\|_{U^{s}}:=\Bigg(\sum_{x\in\mathbb{Z},  \boldsymbol{h}\in\mathbb{Z}^{s}} \Delta_{\boldsymbol{h}} f(x)\Bigg)^{2^{-s}}.
		\end{equation*}
			
		\item\label{Gowers norm,(iii)} Let $s,n\geq 1$, $\boldsymbol{h}\in\mathbb{Z}^{s}$, $\boldsymbol{x}\in\mathbb{Z}^{n}$, and  $f:\mathbb{Z}^{n}\rightarrow\mathbb{C}$. To simplify expressions involving differences along the $\boldsymbol{e_{n}}$ direction, we introduce the notation
		\begin{equation*}
			\Delta_{\boldsymbol{h} | \boldsymbol{e_{n}}} f(\boldsymbol{x}):=\left(\Delta_{\boldsymbol{h}}(f)_{\widehat{\boldsymbol{x}}^{n}}\right)\left(x_{n}\right),\quad
			\partial_{\boldsymbol{h} |\boldsymbol{e_{n}}} f(\boldsymbol{x}):=\left(\partial_{\boldsymbol{h}}(f)_{\widehat{\boldsymbol{x}}^{n}}\right)\left(x_{n}\right).
		\end{equation*}
		Observe that this notation satisfies the identity
		\begin{equation*}
			\left(\Delta_{\boldsymbol{h}|\boldsymbol{e_{n}}}f\right)_{\widehat{\boldsymbol{x}}^{n}}=\Delta_{\boldsymbol{h}}\left(f\right)_{\widehat{\boldsymbol{x}}^{n}}.
		\end{equation*}
	\end{enumerate}
\end{defnprop}

We design the following self-improving process to prove Theorem \ref{Lambda_{N}-inverse with phase} for the case $n\geq 2$.
\begin{proposition}\label{degree-lowering}
	Let $n\geq 2$, $k\geq 0$, and $\boldsymbol{m}$ be a strictly increasing $(n+k)$-tuple of positive integers. Consider the following three statements:
	\begin{enumerate}[(a)]
		\item\label{degree-lowering,(a)} There exist $C>0$ and $s\geq 2$, both depending only on $n$, $k$, and $\boldsymbol{m}$, such that the following holds. Let $\delta\in(0,1/10)$, $N\geq 1$,   $\alpha_{1},\ldots,\alpha_{k}:\mathbb{Z}^{n}\rightarrow\mathbb{T}$ be phase functions, and $f_{0},\ldots, f_{n}:\mathbb{Z}^{n}\rightarrow\mathbb{C}$ be 1-bounded functions. If $N \geq \delta^{-C}$ and
		\begin{equation*}
			\left|\Lambda_{N;\boldsymbol{m}}^{\alpha_{1},\ldots,\alpha_{k}}\left(f_{0}, \ldots, f_{n}\right)\right| \geq \delta,
		\end{equation*}
		then 	
		\begin{equation*}
			\underset{\widehat{\boldsymbol{x}}^{n} \in \prod_{1 \leq i \leq n-1}\left[N^{m_{i}}\right]}{\mathbf{E}}\left\|\left(f_{n}\right)_{\widehat{\boldsymbol{x}}^{n}}\chi_{\left[2N^{m_{n}}\right]}\right\|_{U^{s}}^{2^{s}} \geq \delta^{C} N^{m_{n}(s+1)}.
		\end{equation*}
		
		\item\label{degree-lowering,(b)}
		There exist $C>0$ and $s'\geq 2$, both depending only on $n$, $k$, and $\boldsymbol{m}$, such that the following holds. Let $\delta\in(0,1/10)$, $N\geq 1$,   $\alpha_{1},\ldots,\alpha_{k}:\mathbb{Z}^{n}\rightarrow\mathbb{T}$ be phase functions, and $f_{0},\ldots, f_{n}:\mathbb{Z}^{n}\rightarrow\mathbb{C}$ be 1-bounded functions. If $N \geq \delta^{-C}$ and
		\begin{equation*}
			\left|\Lambda_{N;\boldsymbol{m}}^{\alpha_{1},\ldots,\alpha_{k}}\left(f_{0}, \ldots, f_{n}\right)\right| \geq \delta,
		\end{equation*}
		then there exists $1\leq q\leq \delta^{-C}$ such that for any $1\leq i\leq n-1$ and any $1\leq L\leq \delta^{C}N$, 	
		\begin{equation*}
			\underset{\substack{\widehat{\boldsymbol{x}}^{i} \in 	\prod_{1 \leq j \leq n, j \neq i}\left[ N^{m_{j}}\right],\\ \boldsymbol{h}\in\left[\pm 2N^{m_{n}}\right]^{s^{\prime}-2}}}{\mathbf{E}}\left\|\mathbf{E}\left(\left.\left(\Delta_{\boldsymbol{h}|\boldsymbol{e_{n}}}f_{i}\right)_{\widehat{\boldsymbol{x}}^{i}}\chi_{\left[2N^{m_{i}}\right]} \right\rvert \mathcal{B}_{\left(q^{m_{i}}  L^{m_{i}},q^{m_{i}}\right)}\right)\right\|_{2}^{2} \geq \delta^{C} N^{m_{i}}.
		\end{equation*}
		(The difference operator above disappears when $s^{\prime}=2$.)
		
		\item\label{degree-lowering,(c)} 		
		There exist $C>0$ and $s''\geq 2$, both depending only on $n$, $k$, and $\boldsymbol{m}$, such that the following holds. Let $\delta\in(0,1/10)$, $N\geq 1$,   $\alpha_{1},\ldots,\alpha_{k}:\mathbb{Z}^{n}\rightarrow\mathbb{T}$ be phase functions, and $f_{0},\ldots, f_{n}:\mathbb{Z}^{n}\rightarrow\mathbb{C}$ be 1-bounded functions.  If $N \geq \delta^{-C}$ and
		\begin{equation*}
			\left|\Lambda_{N;\boldsymbol{m}}^{\alpha_{1},\ldots,\alpha_{k}}\left(f_{0}, \ldots, f_{n}\right)\right| \geq \delta,
		\end{equation*}
		then there exists $1\leq q\leq \delta^{-C}$ such that for any $1\leq L\leq \delta^{C}N$, 	
		\begin{equation*}
			\underset{\substack{\widehat{\boldsymbol{x}}^{n} \in 	\prod_{1 \leq j \leq n-1}\left[ N^{m_{j}}\right],\\ \boldsymbol{h}\in\left[\pm 2N^{m_{n}}\right]^{s^{\prime\prime}-2} }}{\mathbf{E}}\left\|\mathbf{E}\left(\left.\left(\Delta_{\boldsymbol{h}|\boldsymbol{e_{n}}}f_{n}\right)_{\widehat{\boldsymbol{x}}^{n}}\chi_{\left[2N^{m_{n}}\right]} \right\rvert \mathcal{B}_{\left(q^{m_{n}}  L^{m_{n}},q^{m_{n}}\right)}\right)\right\|_{2}^{2} \geq \delta^{C} N^{m_{n}}.
		\end{equation*}
		(The difference operator above disappears when $s^{\prime\prime}=2$.)
	\end{enumerate}
	Then the following implications hold:
	\begin{enumerate}[(i)]
		\item If \eqref{degree-lowering,(a)} holds for some $s$, then \eqref{degree-lowering,(b)} holds for $s^{\prime}=s$.
		
		\item  If \eqref{degree-lowering,(b)} holds for some $s^{\prime}$, then \eqref{degree-lowering,(c)} holds for $s^{\prime\prime}=s^{\prime}$.
		
		\item If \eqref{degree-lowering,(c)} holds for some $s^{\prime\prime}$, then  \eqref{degree-lowering,(a)} holds for $s=s^{\prime\prime}-1$.
	\end{enumerate}
\end{proposition}

\begin{remark}
	It is well known that standard PET induction, together with concatenation results, yields global Gowers norm control for   $\Lambda_{N;\boldsymbol{m}}^{\alpha_{1},\ldots,\alpha_{k}}$; that is, \eqref{degree-lowering,(a)} holds for some $s \ge 2$. Readers interested in the details may consult \cite{KKL24 b, Pel19} for a comprehensive  presentation. Moreover, Proposition \ref{degree-lowering} is self-improving: once \eqref{degree-lowering,(a)} holds for some $s\geq2$, then it holds with $s$ replaced by $s-1$. By iterating this reduction, we can lower the order of the Gowers norm in \eqref{degree-lowering,(a)} until we reach $s=2$. At that point, both \eqref{degree-lowering,(b)} and \eqref{degree-lowering,(c)} are satisfied with the parameters $s'=s''=2$. We therefore obtain Theorem \ref{Lambda_{N}-inverse with phase} for the case $n\geq 2$.
\end{remark}

Next, we show that Proposition \ref{degree-lowering} follows from the case $n=1$ of Theorem \ref{Lambda_{N}-inverse with phase}; the proof of this case will be provided later.

Two techniques will be frequently used. The first is the so-called ``stashing'' trick, which, roughly speaking, allows one to replace the function $f_{i}$ (for $1\leq i\leq n$) appearing in conclusions drawn from the assumption
 \begin{equation*}
	\left|\Lambda_{N;\boldsymbol{m}}^{\alpha_{1},\ldots,\alpha_{k}}\left(f_{0},\ldots,f_{n}\right)\right|\geq \delta,
\end{equation*}
by its dual function $F^{(i)}$, defined as
\begin{equation}\label{dual function}
	\begin{split}
		F^{(i)}(\boldsymbol{x})=\underset{r \in[N]}{\mathbf{E}} &f_{0}\left(\boldsymbol{x}-r^{m_{i}} \boldsymbol{e_{i}}\right) \prod_{1 \leq j \leq n, j \neq i} f_{j}\left(\boldsymbol{x}+r^{m_{j}} \boldsymbol{e}_{\boldsymbol{j}}-r^{m_{i}} \boldsymbol{e}_{\boldsymbol{i}}\right)\\
		& \qquad \times e\Bigg(\sum_{1\leq j\leq k}\alpha_{j}\left(\boldsymbol{x}-r^{m_{i}}\boldsymbol{e_{i}}\right)r^{m_{n+j}}\Bigg).
	\end{split}
\end{equation}
This kind of replacement is achieved by using the Cauchy--Schwarz inequality. Here is an example.

\begin{example}\label{stashing trick example}
	Under the same conditions as in Proposition \ref{degree-lowering}\eqref{degree-lowering,(a)}, and without loss of generality, we may further assume that each
	$f_{i}$ is supported in $\prod_{1\leq j\leq n}[2^{\mathbf{I}(j=i)}N^{m_{j}}]$ for $0\leq i\leq n$. We can then rewrite $\Lambda_{N;\boldsymbol{m}}\left(f_{0},\ldots,f_{n}\right)$ as
	\begin{equation*}
		\Lambda_{N;\boldsymbol{m}}^{\alpha_{1},\ldots,\alpha_{k}}\left(f_{0},\ldots,f_{n}\right)=\frac{1}{N^{m_{1}+\cdots+m_{n}}}\sum_{\boldsymbol{x}}f_{n}(\boldsymbol{x})F^{(n)}(\boldsymbol{x}),
	\end{equation*}
where $F^{(n)}$ is given by \eqref{dual function}. Since $F^{(n)}$ is 1-bounded and supported in $\prod_{1\leq j\leq n}[2^{\mathbf{I}(j=n)}N^{m_{j}}]$, applying the Cauchy--Schwarz inequality yields
	\begin{equation*}
		\begin{split}
			\left|\Lambda_{N;\boldsymbol{m}}^{\alpha_{1},\ldots,\alpha_{k}}\left(f_{0},\ldots,f_{n}\right)\right|^{2}&\leq\frac{2}{N^{m_{1}+\cdots+m_{n}}}\sum_{\boldsymbol{x}}\overline{F^{(n)}(\boldsymbol{x})}F^{(n)}(\boldsymbol{x})\\
			&=2\Lambda_{N;\boldsymbol{m}}^{\alpha_{1},\ldots,\alpha_{k}}\left(f_{0},\ldots,f_{n-1},\overline{F^{(n)}}\right).
		\end{split}
	\end{equation*}
	Therefore,
	\begin{equation*}
		\left|\Lambda_{N;\boldsymbol{m}}^{\alpha_{1},\ldots,\alpha_{k}}\left(f_{0},\ldots,f_{n}\right)\right|\geq \delta \Rightarrow\left|\Lambda_{N;\boldsymbol{m}}^{\alpha_{1},\ldots,\alpha_{k}}\left(f_{0},\ldots,f_{n-1},\overline{F^{(n)}}\right)\right|\geq \frac{\delta^{2}}{2}.
	\end{equation*}
	Hence, by Proposition \ref{degree-lowering}\eqref{degree-lowering,(a)}, there exist constants $C>0$ and $s\geq 2$ (the dependencies are omitted here for simplicity) such that
	\begin{equation*}
		\underset{\widehat{\boldsymbol{x}}^{n} \in \prod_{1 \leq i \leq n-1}\left[N^{m_{i}}\right]}{\mathbf{E}}\left\|\left(F^{(n)}\right)_{\widehat{\boldsymbol{x}}^{n}}\right\|_{U^{s}}^{2^{s}} \geq \delta^{C} N^{m_{n}(s+1)}.
	\end{equation*}
In later proofs, we will simply refer to this kind of argument as the stashing trick for $f_{i}$.
\end{example}

The second is the so-called ``dual-difference interchange'', which, roughly speaking, aims to exchange the order of the expectation operator (which usually appears after applying the stashing trick) and the difference operator (which usually arises in the expansion of Gowers norms). This kind of exchange is achieved by using the Cauchy--Schwarz inequality as well. Following the remark below, we present several results.

\begin{remark}
Van der Corput’s inequality is sometimes  more flexible to use than a direct application of the Cauchy--Schwarz inequality. Below we present a useful variant that follows directly from \cite[Lemma 3.1]{Pre17}.

	Let $\delta\in(0,1)$, $M\geq 1$, $I\subset\mathbb{Z}$ be an interval of length $M$, and $(\mathcal{A},\sigma)$ be a discrete probability space. Suppose that for every $\alpha\in\mathcal{A}$, $f_{\alpha }:\mathbb{Z}\rightarrow \mathbb{C}$ is 1-bounded. If $M\geq 10\delta^{-2}$ and
	\begin{equation*}
		\sum_{\alpha\in\mathcal{A}} \sigma(\alpha )\left|\underset{y\in I}{\mathbf{E}} f_{\alpha }(y)\right| \geq \delta,
	\end{equation*}
	then for any $1\leq H\leq (\delta^{2} / 4)  M$,
	\begin{equation*}
		\operatorname{Re}\Bigg(\sum_{\alpha\in \mathcal{A},h\in \mathbb{Z}} \sigma(\alpha)\mu_{H}(h) \underset{y \in I }{\mathbf{E}} \Delta_{h} f_\alpha (y)\Bigg) \geq \frac{\delta^{2}}{4},
	\end{equation*}
	where $\mu_{H}$ is a probability measure on $\mathbb{Z}$ (the Fej\'{e}r kernel) defined by
	\begin{equation*}
		\mu_{H}(x):=\frac{H-\min (H,|x|)}{H^{2}}, \quad x \in \mathbb{Z}.
	\end{equation*}
\end{remark}

\begin{lemma}\label{dual-difference interchange}
	\leavevmode

	\begin{enumerate}[(i)]
		\item\label{dual-difference interchange,(i)} \textup{(\cite[Lemma 3.6]{KKL24 a})} Given $s\geq 2$, there exists a constant  $C=C(s)>0$ such that the following holds. Let \(\delta\in(0,1/10)\), \(N\ge 1\), and \(\mathcal{A}\) a nonempty finite set. Suppose that for every \(\alpha\in\mathcal{A}\), \(f_\alpha:\mathbb{Z}\to\mathbb{C}\) is \(1\)-bounded and supported in \([N]\). Set \(F := \mathbf{E}_{\alpha \in \mathcal{A}} f_\alpha\).
		If $\|F\|_{U^{s}}^{2^{s}} \geq \delta N^{s+1}$ and $N\geq \delta^{-C}$, then
		\begin{equation*}
			\underset{\boldsymbol{h}\in[\pm N]^{s-2}}{\mathbf{E}}\left\|\underset{\alpha \in \mathcal{A}}{\mathbf{E}} \Delta_{\boldsymbol{h}} f_{\alpha}\right\|_{U^{2}}^{4} \geq  \delta^{C} N^{3}.
		\end{equation*}
		
		\item \label{dual-difference interchange,(ii)}
		Given $s\geq 1$, there exists a constant  $C=C(s)>0$ such that the following holds. Let \(\delta\in(0,1/10)\) and \(q,L,N_{1},N_{2}\ge 1\) satisfy
		\[
		N_{1},N_{2}\ge \delta^{-C},\qquad L\ge \delta N_{1}.
		\]
		Let \(\mathcal{A}\) be a nonempty finite set. Suppose that for every \(\alpha\in\mathcal{A}\), \(f_{\alpha}:\mathbb{Z}^{2}\to\mathbb{C}\) is \(1\)-bounded and supported in \([N_{1}]\times[N_{2}]\). Set \(F:=\mathbf{E}_{\alpha \in \mathcal{A}} f_{\alpha}\). If
		\begin{equation*}
			\underset{y \in\left[N_{2}\right], \boldsymbol{h}\in [\pm N_{2}]^{s}}{\mathbf{E}}\left\|\mathbf{E}\left(\left.\left(\Delta_{\boldsymbol{h} \mid \boldsymbol{e_{2}}} F\right)_{y} \right\rvert \mathcal{B}_{(q L, q)}\right)\right\|_{2}^{2} \geq \delta N_{1},
		\end{equation*}
		then
		\begin{equation*}
			\underset{y \in\left[N_{2}\right],\boldsymbol{h}\in [\pm N_{2}]^{s}}{\mathbf{E}}\left| \underset{x \in \left[N_{1}\right],\alpha \in \mathcal{A}}{\mathbf{E}} \Delta_{\boldsymbol{h} | \boldsymbol{e_{2}}} f_{\alpha}(x, y) \right| \geq   \delta^{C} .
		\end{equation*}
		
	\end{enumerate}
\end{lemma}
\begin{proof}
	We prove \eqref{dual-difference interchange,(ii)} by induction on $s$. For the base case $s=1$, by expanding the $\ell^{2}$-norm,  using van der Corput's inequality and changing variables, we have
	\begin{equation*}
		\Bigg|\underset{\substack{y \in\left[N_{2}\right],\\ \left|h_{1}\right|\leq N_{2}}}{\mathbf{E}} \sum_{x,k} \mu_{K}(k) F(x, y) \overline{F(x+q k, y) F\left(x, y+h_{1}\right)} F\left(x+q k, y+h_{1}\right)\!\Bigg|\! \gtrsim \delta N_{1}
	\end{equation*}
	for some $K\asymp \delta L$. Further,
	\begin{equation*}
		\underset{\alpha, \alpha^{\prime} \in \mathcal{A}}{\mathbf{E}} \sum_{x,y, k}\left|\underset{\left|h_{1}\right| \leq N_{2}}{\mathbf{E}} f_{\alpha}\left(x, y+h_{1}\right) \overline{f_{\alpha^{\prime}}\left(x+qk, y+h_{1}\right)}\right| \gtrsim \delta N_{1}N_{2} K \gtrsim \delta^{3} N_{1}^{2}N_{2}.
	\end{equation*}
	By positivity and changing variables,
	\begin{equation*}
		\underset{\substack{x,x^{\prime}\in\left[N_{1}\right],\\ \alpha, \alpha^{\prime} \in \mathcal{A}}}{\mathbf{E}} \sum_{y}\left|\underset{\left|h_{1}\right| \leq N_{2}}{\mathbf{E}} f_{\alpha}\left(x, y+h_{1}\right) \overline{f_{\alpha^{\prime}}\left(x^{\prime}, y+h_{1}\right)}\right| \gtrsim \delta^{3}N_{2}.
	\end{equation*}
	By van der Corput's inequality and changing variables again, we obtain
	\begin{equation*}
		\Bigg|\underset{\substack{y \in \left[N_{2}\right], \\ x,x^{\prime} \in  \left[N_{1}\right], \\ \alpha, \alpha^{\prime} \in \mathcal{A}}}{\mathbf{E}} \sum_{h_{1}} \mu_{H}\left(h_{1}\right) \Delta_{h_{1}|\boldsymbol{e_{2}}} f_{\alpha}(x, y) \overline{\Delta_{h_{1} | \boldsymbol{e_{2}}} f_{\alpha^{\prime}}\left(x^{\prime}, y\right)}\Bigg| \gtrsim \delta^{6}
	\end{equation*}
	for some $H\asymp\delta^{6} N_{2}$. Furthermore,
	\begin{equation*}
		\underset{y\in \left[N_{2}\right],\left|h_{1}\right| \leq  N_{2}}{\mathbf{E}}\left|\underset{x \in \left[N_{1}\right], \alpha \in \mathcal{A}}{\mathbf{E}} \Delta_{h_{1} | \boldsymbol{e_{2}}} f_{\alpha}(x, y)\right|^{2} \gtrsim \delta^{12}.
	\end{equation*}
	The desired result then follows from the popularity principle.
	
	 Assuming the statement holds for $s$, we treat the case $s+1$. We first rewrite the core condition as
	\begin{equation*}
		\underset{\substack{y \in\left[N_{2}\right],|h_{s+1}|\leq N_{2},\\ \boldsymbol{h}\in[\pm
		N_{2}]^{s}}}{\mathbf{E}}\left\|\mathbf{E}\left(\left. \!\!\!\left(\Delta_{\boldsymbol{h} \mid \boldsymbol{e_{2}}} \left(\underset{\alpha, \alpha^{\prime} \in \mathcal{A}}{\mathbf{E}} f_{\alpha} \overline{f_{\alpha^{\prime}}\left(\cdot+h_{s+1} \boldsymbol{e}_{2}\right)}\right)\right)_{y} \right\rvert \mathcal{B}_{(q L, q)}\!\!\right)\right\|_{2}^{2}  \geq \delta N_{1}.
	\end{equation*}
	By the popularity principle (for $h_{s+1}$) and the inductive hypothesis, we have
	\begin{equation*}
		\underset{\substack{y \in\left[N_{2}\right],|h_{s+1}|\leq N_{2},\\ \boldsymbol{h}\in[\pm N_{2}]^{s}}}{\mathbf{E}}\Bigg|\underset{\substack{x\in \left[N_{1}\right] \\ \alpha, \alpha^{\prime} \in \mathcal{A}}}{\mathbf{E}} \Delta_{\boldsymbol{h}| \boldsymbol{e_{2}}} f_{\alpha}(x, y) \overline{\Delta_{\boldsymbol{h} | \boldsymbol{e_{2}}} f_{\alpha^{\prime}}\left(x, y+h_{s+1}\right)}\Bigg|  \geq   \delta^{C}
	\end{equation*}
	for some $C=C(s)>0$. Squaring both sides above yields
	\begin{equation*}
		\begin{split}
			\underset{\substack{y 	\in\left[N_{2}\right],|h_{s+1}|\leq N_{2},\\ x,x^{\prime}\in \left[N_{1}\right],\\ \alpha,\alpha^{\prime},\beta,\beta^{\prime}\in\mathcal{A},\\ \boldsymbol{h}\in[\pm N_{2}]^{s}}}{\mathbf{E}}&\Delta_{\boldsymbol{h} | \boldsymbol{e_{2}}} f_{\alpha}(x, y)\overline{\Delta_{\boldsymbol{h}| \boldsymbol{e_{2}}} f_{\beta}(x^{\prime}, y)} \\ &\times\overline{\Delta_{\boldsymbol{h} | \boldsymbol{e_{2}}} f_{\alpha^{\prime}}\left(x, y+h_{s+1}\right)}\Delta_{\boldsymbol{h} | \boldsymbol{e_{2}}} f_{\beta^{\prime}}\left(x^{\prime}, y+h_{s+1}\right)  \geq  \delta^{2C}.
		\end{split}
	\end{equation*}	
	Further,
	\begin{equation*}
		\underset{\substack{x,x^{\prime}\in \left[N_{1}\right],\\ \alpha^{\prime},\beta^{\prime}\in\mathcal{A},\\ \boldsymbol{h}\in[\pm N_{2}]^{s}}}{\mathbf{E}}\sum_{y} \left|\underset{|h_{s+1}|\leq  N_{2}}{\mathbf{E}}\!\!\overline{\Delta_{\boldsymbol{h} | \boldsymbol{e_{2}}} f_{\alpha^{\prime}}\left(x, y+h_{s+1}\right)} \Delta_{\boldsymbol{h} | \boldsymbol{e_{2}}} f_{\beta^{\prime}}\left(x^{\prime}, y+h_{s+1}\right)\right| \geq\delta^{2C}N_{2}.
	\end{equation*}	
The proof is then completed by applying the same reasoning as in the base case.
\end{proof}

\begin{remark}
	We briefly explain how the expectation and difference operators ``commute'' in the setting of Lemma \ref{dual-difference interchange}. We start from \eqref{dual-difference interchange,(i)}. One can naturally obtain from the definition of Gowers norms that
	\begin{equation*}
		\|F\|_{U^{s}}^{2^{s}} \geq \delta N^{s+1} \Rightarrow \underset{\boldsymbol{h} \in[ \pm N]^{s-2}}{\mathbf{E}}\left\|\Delta_{\boldsymbol{h}} F\right\|_{U^{2}}^{4} \gtrsim_{s} \delta N^{3}.
	\end{equation*}
Substituting  \(F=\mathbf{E}_{\alpha \in \mathcal{A}} f_\alpha\) gives
	\begin{equation*}
		\underset{\boldsymbol{h} \in[ \pm N]^{s-2}}{\mathbf{E}}\left\|\Delta_{\boldsymbol{h}}\left(\underset{\alpha \in \mathcal{A}}{\mathbf{E}} f_{\alpha}\right)\right\|_{U^{2}}^{4} \gtrsim_{s} \delta N^{3}.
	\end{equation*}
Comparing this with the conclusion of \eqref{dual-difference interchange,(i)} clarifies the nature of the ``dual-difference interchange'' principle.
		
	The situation in \eqref{dual-difference interchange,(ii)} is less straightforward. In fact, it shares a similar flavor with the following more intuitive statement, whose proof is omitted as it is analogous to that of \eqref{dual-difference interchange,(ii)}. Given $s\geq 1$, there exists a constant $C=C(s)>0$  such that the following holds. Let \(\delta\in(0,1/10)\) and \(N_{1},N_{2}\ge 1\).
	Let \(\mathcal{A}\) be a nonempty finite set. Suppose that \(\varphi:\mathbb{Z}\to\mathbb{C}\) is \(1\)-bounded, and that for every \(\alpha\in\mathcal{A}\), \(f_{\alpha}:\mathbb{Z}^{2}\to\mathbb{C}\) is \(1\)-bounded and supported in \([N_{1}]\times[N_{2}]\). Set \(F= \mathbf{E}_{\alpha\in\mathcal{A}} f_{\alpha}\). If  $N_{1},N_{2}\geq  \delta^{-C}$ and
	\begin{equation*}
		\underset{y\in[N_{2}],\boldsymbol{h}\in[\pm N_{2}]^{s}}{\mathbf{E}}\left|\underset{x\in[N_{1}]}{\mathbf{E}} \varphi(x) \Delta_{\boldsymbol{h}| \boldsymbol{e_{2}}} F(x, y)\right| \geq \delta,
	\end{equation*}
	then
	\begin{equation*}
		\underset{y\in[N_{2}],\boldsymbol{h}\in[\pm N_{2}]^{s}}{\mathbf{E}}\left| \underset{x\in [N_{1}], \alpha \in \mathcal{A}}{\mathbf{E}} \Delta_{\boldsymbol{h}|\boldsymbol{e_{2}}} f_{\alpha}(x, y) \right| \geq  \delta^{C} .
	\end{equation*}
\end{remark}

The key point in Lemma \ref{dual-difference interchange}\eqref{dual-difference interchange,(ii)} is that the difference and conditional expectation operators act on different coordinates. The situation becomes quite different when they act on the same coordinate; see the following lemma. 	

\begin{lemma}\label{Claim 3}
	Let \(\delta\in(0,1/10)\) and \(s,q,L,N_{1},N_{2}\ge 1\) satisfy
	\[
	N_{2}\ge \delta^{-3},\qquad L\ge \delta N_{2},
	\]
	and let \(f:\mathbb{Z}^{2}\to\mathbb{C}\) be a \(1\)-bounded function supported in \([N_{1}]\times[N_{2}]\). If
	\begin{equation*}
		\underset{x \in\left[N_{1}\right], \boldsymbol{h}\in [\pm N_{2}]^{s}}{\mathbf{E}}\left\|\mathbf{E}\left(\left.\left(\Delta_{\boldsymbol{h} \mid \boldsymbol{e_{2}}} f\right)_{x} \right\rvert \mathcal{B}_{(q L, q)}\right)\right\|_{2}^{2} \geq \delta N_{2},
	\end{equation*}
	then
	\begin{equation*}
		\underset{x\in[N_{1}]}{\mathbf{E}}\left\| f_{x}\right\|_{U^{s+1}}^{2^{s+1}}\gtrsim_{s} \delta^{3} N_{2}^{s+2}.
	\end{equation*}
\end{lemma}

\begin{proof}
	By expanding the $\ell^{2}$-norm, using van der Corput's inequality and changing variables (as in the proof of Lemma \ref{dual-difference interchange}\eqref{dual-difference interchange,(ii)}), we obtain
	\begin{equation*}
		\Bigg|\underset{x\in[N_{1}],\boldsymbol{h}\in[\pm N_{2}]^{s}}{\mathbf{E}} \sum_{y,k}\mu_{K}(k)\Delta_{\boldsymbol{h},qk} f_{x}(y)\Bigg| \gtrsim \delta N_{2},
	\end{equation*}
	where $K\asymp \delta L$. That is
	\begin{equation*}
		\underset{x\in[N_{1}]}{\mathbf{E}}\sum_{k} \mu_{K}(k)\left\|\Delta_{qk} f_{x}\right\|_{U^{s}}^{2^{s}}\gtrsim_{s} \delta N_{2}^{s+1}.
	\end{equation*}
	The result then follows easily. 	
\end{proof}

We also need the following $U^{2}$-inverse theorem.
\begin{lemma}\textup{(\cite[Lemma A.1]{PP24})}\label{U^2 inverse}
	Let \(\delta\in(0,1)\) and \(N\ge 1\). Let \(f:\mathbb{Z}\to\mathbb{C}\) be a \(1\)-bounded function supported in an interval of length \(N\). If $\|f\|_{U^{2}}^{4} \geq \delta N^{3}$, then there exists $\alpha\in\mathbb{T}$ such that
	\begin{equation*}
		\Bigg|\sum_{x} f(x) e(\alpha x)\Bigg|\geq \delta N.
	\end{equation*}
\end{lemma}

We prove Proposition \ref{degree-lowering} by induction on the dimension $n$. We first establish the inductive step.

\begin{proof}[Proof of Proposition \ref{degree-lowering} for $n\geq 3$ assuming the case $n-1$]	
	
	\leavevmode
	
	We need to prove the three implications in Proposition \ref{degree-lowering}. Throughout the proofs, all constants $C$ depend only on $n,k,\boldsymbol{m}$ and may vary from line to line. For each $0\leq i\leq n$, we may assume without loss of generality that $f_{i}$ is supported in $\prod_{1\leq j\leq n}[2^{\mathbf{I}(j=i)}N^{m_{j}}]$. Recall also that at this stage Theorem \ref{Lambda_{N}-inverse with phase} holds for $n-1$ (see the remark following Proposition \ref{degree-lowering}).

	\textbf{Implication 1}: if \eqref{degree-lowering,(a)} holds for some $s\geq 2$, then \eqref{degree-lowering,(b)} holds for $s^{\prime}=s$.

	By the stashing trick for $f_{n}$, we obtain from \eqref{degree-lowering,(a)} that if $N$ is large, then
	\begin{equation*}
		\underset{\widehat{\boldsymbol{x}}^{n} \in \prod_{1 \leq i \leq n-1}\left[N^{m_{i}}\right]}{\mathbf{E}}\left\|\left(F^{(n)}\right)_{\widehat{\boldsymbol{x}}^{n}}\right\|_{U^{s}}^{2^{s}} \geq \delta^{C} N^{m_{n}(s+1)},
	\end{equation*}
	where $F^{(n)}$ is given by \eqref{dual function}.
	By the popularity principle, Lemma \ref{dual-difference interchange}\eqref{dual-difference interchange,(i)}, Lemma \ref{U^2 inverse}, and a change of variables, for all $\boldsymbol{h}$ in some $E\subset [\pm 2N^{m_{n}}]^{s-2}$ of density $\geq \delta^{C}$ and for all $\widehat{\boldsymbol{x}}^{n}$ in some $E'\subset\prod_{1 \leq i \leq n-1}[N^{m_{i}}]$ (depending on $\boldsymbol{h}$) of density $\geq \delta^{C}$, there is $\beta =\beta (\boldsymbol{h},\widehat{\boldsymbol{x}}^{n})\in\mathbb{T}$ such that
	\begin{equation*}
		\begin{split}
			\Bigg|\underset{x_{n}\in \left[N^{m_{n}}\right], r\in [N]}{\mathbf{E}}
			&\Delta_{\boldsymbol{h} | \boldsymbol{e}_{\boldsymbol{n}}} f_{0}(\boldsymbol{x}) \prod_{1 \leq j \leq n-1} \Delta_{\boldsymbol{h} | \boldsymbol{e}_{\boldsymbol{n}}} f_{j}\left(\boldsymbol{x}+r^{m_{j}} \boldsymbol{e}_{\boldsymbol{j}}\right)  \\
			&\times e \Bigg(\beta x_{n}+\beta r^{m_{n}}+\sum_{1\leq j\leq k}\partial_{\boldsymbol{h}|\boldsymbol{e_{n}}}\alpha_{j}(\boldsymbol{x}) r^{m_{n+j}}\Bigg) \Bigg|\geq \delta^{C}.
		\end{split}
	\end{equation*}
	Further, by the triangle inequality for $x_{n}$, summing both sides over $\widehat{\boldsymbol{x}}^{n}\in E'$, and the popularity principle, we find a subset $E^{\prime\prime}\subset[ N^{m_{n}}]$ (depending on $\boldsymbol{h}$) of density $\geq \delta^{C}$ such that for all $x_{n}\in E^{\prime\prime}$,
	\begin{align*}
		&\underset{\widehat{\boldsymbol{x}}^{n} \in E^{\prime}}{\mathbf{E}}\Bigg|\underset{r \in [N]}{\mathbf{E}} \!\prod_{1 \leq j \leq n-1} \!\!\!\Delta_{\boldsymbol{h} | \boldsymbol{e_{n}}} f_{j}\left(\boldsymbol{x}+r^{m_{j}} \boldsymbol{e}_{\boldsymbol{j}}\right)
		e\Biggl(\beta r^{m_{n}}+\sum_{1\leq j\leq k}\partial_{\boldsymbol{h}|\boldsymbol{e_{n}}}\alpha_{j}(\boldsymbol{x}) r^{m_{n+j}}\Biggr)\Bigg| \\
		&\ \geq \delta^{C}.
	\end{align*}
	For fixed $\boldsymbol{h}$ and $x_{n}$, we can remove the absolute value by introducing a suitable $1$-bounded factor $\psi$ as follows
	\begin{equation*}
		\begin{split}
			\underset{\substack{\widehat{\boldsymbol{x}}^{n} \in \prod_{1 \leq i \leq n-1}\left[N^{m_{i}}\right],\\ r\in [N]}}{\mathbf{E}} &(\psi\chi_{E^{\prime}})(\widehat{\boldsymbol{x}}^{n})\prod_{1 \leq j \leq n-1} \Delta_{\boldsymbol{h} | \boldsymbol{e_{n}}} f_{j}\left(\boldsymbol{x}+r^{m_{j}} \boldsymbol{e}_{\boldsymbol{j}}\right) \\
			&\times e\Biggl(\beta r^{m_{n}}+\sum_{1\leq j\leq k}\partial_{\boldsymbol{h}|\boldsymbol{e_{n}}}\alpha_{j}(\boldsymbol{x}) r^{m_{n+j}}\Biggr) \geq \delta^{C}.
		\end{split}
	\end{equation*}
	Hence, by Theorem \ref{Lambda_{N}-inverse with phase} with $n-1$, there exists $1\leq q \leq \delta^{-C}$ (depending on $\boldsymbol{h}$ and $x_{n}$) such that for any $1 \leq i \leq n-1$ and any $1\leq L \leq \delta^{C} N$,
	\begin{equation*}
		\underset{\widehat{\boldsymbol{x}}^{i,n} \in \prod_{1 \leq j \leq n-1, j \neq i}\left[ N^{m_{j}}\right]}{\mathbf{E}}\left\|\mathbf{E}\left(\left.\left(\Delta_{\boldsymbol{h}|\boldsymbol{e_{n}}}f_{i}\right)_{\widehat{\boldsymbol{x}}^{i}} \right\rvert \mathcal{B}_{\left(q^{m_{i}}  L^{m_{i}},q^{m_{i}}\right)}\right)\right\|_{2}^{2} \geq \delta^{C} N^{m_{i}}.
	\end{equation*}
	Since there are at most $\delta^{-C}$ possible values for $q$, there exists $1\leq q\leq \delta^{-C}$ independent of $\boldsymbol{h}$ and $x_{n}$ such that
	\begin{equation*}
		\underset{\substack{\widehat{\boldsymbol{x}}^{i} \in 	\prod_{1 \leq j \leq n, j \neq i}\left[ N^{m_{j}}\right],\\ \boldsymbol{h}\in\left[\pm 2N^{m_{n}}\right]^{s-2}}}{\mathbf{E}}\left\|\mathbf{E}\left(\left.\left(\Delta_{\boldsymbol{h}|\boldsymbol{e_{n}}}f_{i}\right)_{\widehat{\boldsymbol{x}}^{i}}\right\rvert \mathcal{B}_{\left(q^{m_{i}}  L^{m_{i}},q^{m_{i}}\right)}\right)\right\|_{2}^{2} \geq \delta^{C} N^{m_{i}},
	\end{equation*}
	as desired.
	
	\textbf{Implication 2}: if \eqref{degree-lowering,(b)} holds for some $s^{\prime}\geq 2$, then \eqref{degree-lowering,(c)} holds for $s^{\prime\prime}=s^{\prime}$.
	
	By the stashing trick for $f_{1}$, we obtain from $\eqref{degree-lowering,(b)}$ that if $N$ is large, then there exist some $q\geq 1$ and $L\asymp \delta^{C} N^{m_{1}}$ such that
	\begin{equation*}
		\underset{\substack{\widehat{\boldsymbol{x}}^{1} \in \prod_{2 \leq j \leq n}\left[ N^{m_{j}}\right],\\ \boldsymbol{h}\in[\pm 2N^{m_{n}}]^{s^{\prime}-2}}}{\mathbf{E}}\left\|\mathbf{E}\left(\left.\left(\Delta_{\boldsymbol{h}|\boldsymbol{e_{n}}}F^{(1)}\right)_{\widehat{\boldsymbol{x}}^{1}} \right\rvert \mathcal{B}_{(qL,q)}\right)\right\|_{2}^{2} \geq \delta^{C} N^{m_{1}},
	\end{equation*}
	where $F^{(1)}$ is given by \eqref{dual function}.
	By the popularity principle (for $x_{2},\ldots,x_{n-1}$), Lemma \ref{dual-difference interchange}\eqref{dual-difference interchange,(ii)}, and a change of variables, we have
	\begin{equation*}
		\begin{split}
			\underset{\substack{\widehat{\boldsymbol{x}}^{1} \in \prod_{2 \leq j \leq n}\left[ N^{m_{j}}\right],\\ \boldsymbol{h}\in[\pm 2N^{m_{n}}]^{s^{\prime}-2}}}{\mathbf{E}}&\Bigg|\underset{x_{1} \in\left[N^{m_{1}}\right], r \in [N]}{\mathbf{E}} \Delta_{\boldsymbol{h} | \boldsymbol{e}_{\boldsymbol{n}}} f_{0}(\boldsymbol{x}) \prod_{2 \leq j \leq n} \Delta_{\boldsymbol{h} | \boldsymbol{e}_{\boldsymbol{n}}}f_{j}\left(\boldsymbol{x}+r^{m_{j}} \boldsymbol{e}_{\boldsymbol{j}}\right) \\
			&\times e\Biggl(\sum_{1\leq j\leq k}\partial_{\boldsymbol{h}|\boldsymbol{e_{n}}}\alpha_{j}(\boldsymbol{x}) r^{m_{n+j}}\Biggr)\Bigg| \geq\delta^{C}.
		\end{split}
	\end{equation*}
	Therefore, the proof of Implication 2 can be concluded by applying the same reasoning as at the end of the proof of Implication 1.

	\textbf{Implication 3}: if \eqref{degree-lowering,(c)} holds for some $s^{\prime\prime}\geq 2$, then \eqref{degree-lowering,(a)} holds for $s=s^{\prime\prime}-1$.
	
	We derive from $\eqref{degree-lowering,(c)}$ that if $N$ is large, then there exist $q\geq 1$ and $L\asymp \delta^{C} N^{m_{n}}$ such that
	\begin{equation*}
		\underset{\substack{\widehat{\boldsymbol{x}}^{n} \in \prod_{1 \leq j \leq n-1}\left[ N^{m_{j}}\right],\\ \boldsymbol{h}\in[\pm 2N^{m_{n}}]^{s^{\prime\prime}-2}}}{\mathbf{E}}\left\|\mathbf{E}\left(\left.\left(\Delta_{\boldsymbol{h}|\boldsymbol{e_{n}}}f_{n}\right)_{\widehat{\boldsymbol{x}}^{n}} \right\rvert \mathcal{B}_{(qL,q)}\right)\right\|_{2}^{2} \geq \delta^{C} N^{m_{n}}.
	\end{equation*}
	The desired result then follows from Lemma \ref{Claim 3}.
\end{proof}

An almost identical argument shows that the case $n=1$ of Theorem \ref{Lambda_{N}-inverse with phase} implies the base case $n=2$ of Proposition \ref{degree-lowering}. Hence, assuming the former, we have proved Proposition \ref{degree-lowering}.

It remains to prove the case $n=1$ of Theorem \ref{Lambda_{N}-inverse with phase}. The main tools are two technical results from \cite{PP24} together with Weyl’s inequality from higher order Fourier analysis.

\begin{lemma}\label{technical lemma 1}
	\textup{(\cite[Lemma 6.3]{PP24})}
	Given $s\geq 1$, there exists a constant  $C=C(s)>0$ such that the following holds. Let \(\delta\in(0,1/10)\), $N\geq 1$, $\mathcal{H}\subset \mathbb{Z}^{s}$ be a finite set of size $\geq\delta N^{s}$, and $\phi:\mathbb{Z}^{s}\rightarrow\mathbb{T}$ be a phase function.
	Let \(\mathcal{A}\) be a nonempty finite set. Suppose that, for every \(\alpha\in\mathcal{A}\), \(f_{\alpha}:\mathbb{Z}\to\mathbb{C}\) is a \(1\)-bounded function supported in \([N]\). Set \(F:=\mathbf{E}_{\alpha \in \mathcal{A}} f_{\alpha}\). If
	\begin{equation*}
		\underset{\boldsymbol{h} \in \mathcal{H}}{\mathbf{E}}\left|\underset{x \in[N]}{\mathbf{E}} \Delta_{\boldsymbol{h}} F(x) \cdot e(x \phi(\boldsymbol{h}))\right| \geq \delta,
	\end{equation*}
	then
	\begin{equation*}
		\underset{\boldsymbol{h^{0}}, \boldsymbol{h^{1}} \in \mathcal{H}}{\mathbf{E}}\left|\underset{x \in[N], \alpha \in \mathcal{A}}{\mathbf{E}} \Delta_{\boldsymbol{h^{0}}-\boldsymbol{h^{1}}} f_{\alpha}(x) \cdot e\left(x \tilde{\phi}\left(\boldsymbol{h^{0}}, \boldsymbol{h^{1}} \right)\right)\right| \geq \delta^{C},
	\end{equation*}
	where
	\begin{equation*}
		\tilde{\phi}\left(\boldsymbol{h^{0}}, \boldsymbol{h^{1}}\right):=\sum_{\boldsymbol{\omega} \in\{0,1\}^{s}}(-1)^{|\boldsymbol{\omega}|} \phi\left(h_{1}^{\omega_{1}}, \ldots, h_{s}^{\omega_{s}}\right)
	\end{equation*}
	and $|\boldsymbol{\omega}|$ denotes the number of $1$'s in $\boldsymbol{\omega}$.
\end{lemma}

\begin{lemma}\label{technical lemma 2}
	\textup{(\cite[Lemma 6.4]{PP24})}
	Given $s\geq 2$, there exists a constant  $C=C(s)>0$ such that the following holds. Let \(\delta\in(0,1/10)\), $N\geq 1$, $f:\mathbb{Z}\rightarrow\mathbb{C}$ be a 1-bounded function supported in $[N]$, and $\phi_{1},\ldots,\phi_{s}:\mathbb{Z}^{s-1}\rightarrow\mathbb{T}$ be phase functions. If
	
	\begin{equation*}
		\underset{\boldsymbol{h} \in[ \pm N]^{s}}{\mathbf{E}}\left|\underset{x \in[N]}{\mathbf{E}} \Delta_{\boldsymbol{h}} f(x) e\left(x \sum_{1 \leq i \leq s} \phi_{i}\left(\widehat{\boldsymbol{h}}^{i}\right)\right)\right| \geq \delta,
	\end{equation*}
	then	
	\begin{equation*}
		\|f\|_{U^{s+1}}^{2^{s+1}} \geq \delta^{C} N^{s+2}.
	\end{equation*}	
\end{lemma}

\begin{lemma}\label{weyl}
	\textup{(\cite[Lemma 1.1.16]{Tao12})}
	Given $s\geq 1$, there exists a constant  $C=C(s)>0$  such that the following holds. Let  $\delta \in(0,1 / 10)$, $N\geq 1$, and $ \alpha_{1},\ldots,\alpha_{s}\in\mathbb{T}$.
	If $N \geq \delta^{-C}$ and
	\begin{equation*}
		\left|\underset{n \in[N]}{\mathbf{E}} \left(\sum_{1 \leq i \leq s} \alpha_{i} n^{i}\right)\right| \geq \delta,
	\end{equation*}
	then there exists $1\leq q\leq \delta^{-C}$ such that, for any $1\leq i\leq s$, $\|q\alpha_{i}\|\leq\delta^{-C}N^{-i}$.
\end{lemma}

We split the proof of the case $n=1$ of Theorem \ref{Lambda_{N}-inverse with phase} into the following two propositions.
\begin{proposition}\label{remove dependencies}
	Given \(k\geq 1\) and a strictly increasing \((k+1)\)-tuple \((m_{0},\ldots,m_{k})\in\mathbb{N}^{k+1}\), there exists a constant \(C=C(k,m_{0},\ldots,m_{k})>0\)  such that the following holds. Let \(\delta\in(0,1/10)\), \(N\ge 1\), \(\alpha_1,\dots,\alpha_k:\mathbb{Z}\to\mathbb{T}\) be phase functions, and  \(f:\mathbb{Z}\to\mathbb{C}\) be  \(1\)-bounded function supported in \([2N^{m_0}]\).	If \(N \geq \delta^{-C}\) and
	\begin{equation*}
		\underset{x\in \left[N^{m_{0}}\right]}{\mathbf{E}}\Bigg|\underset{r \in[N]}{\mathbf{E}} f\left(x+r^{m_{0}}\right) e\Bigg(\sum_{1 \leq j \leq k} \alpha_{j}(x) r^{m_{j}}\Bigg)\Bigg| \geq \delta,
	\end{equation*}
	then there exist $\beta_{0},\ldots,\beta_{k}\in\mathbb{T}$, which are independent of $x$, such that
	\begin{equation*}
		\underset{x \in\left[N^{m_{0}}\right]}{\mathbf{E}}\Bigg|\underset{r \in[N]}{\mathbf{E}} f\left(x+r^{m_{0}}\right) e\Bigg(\sum_{1 \leq j \leq k} \beta_{j} r^{m_{j}}\Bigg)\Bigg| \geq \delta^{C}.
	\end{equation*}
\end{proposition}

\begin{proof}
	All constants $C$ below depend only on $k,m_{0},\ldots,m_{k}$ and may change from line to line. At the cost of at most $\delta/2$, we may require that, for any $1\leq j\leq k$ and any $x\in [N]$, $\alpha_{j}(x)$ takes values in $\left\{t/T_{i}:0\leq t<T_{j}\right\}$ for some $T_{j}=\lceil2\delta^{-1} k N^{m_{j}}\rceil$.

	By the popularity principle, for all $x\in A$ for some $A\subset [N^{m_{0}}]$ of density $\gtrsim \delta$,
	\begin{equation}\label{remove dependencies,eq1}
		\Bigg|\underset{r \in[N]}{\mathbf{E}} f(x+r^{m_{0}}) e\Bigg(\sum_{1 \leq j \leq k} \alpha_{j}(x) r^{m_{j}}\Bigg)\Bigg| \gtrsim \delta.
	\end{equation}
	Thus, for some 1-bounded factor $\psi$, we have
	\begin{equation*}
		\underset{x\in[N^{m_{0}}],r\in[N]}{\mathbf{E}}(\psi\chi_{A})(x) f(x+r^{m_{0}}) e\Bigg(\sum_{1 \leq j \leq k} \alpha_{j}(x) r^{m_{j}}\Bigg) \gtrsim \delta^{2}.
	\end{equation*}
	By using standard PET induction and concentration results, we deduce from the above that the Gowers norm of $f$ of some order $s=s(k,m_{0},\ldots,m_{k})\geq 2$ is large. Then, by the stashing trick for $f$, we have
	\begin{equation*}
		\|F\|_{U^{s}}^{2^{s}}\geq \delta^{C}N^{m_0(s+1)},
	\end{equation*}
	where
	\begin{equation*}
		F(x):=\underset{r\in[N]}{\mathbf{E}}(\psi\chi_{A})(x-r^{m_{0}})e\Bigg(\sum_{1 \leq j \leq k} \alpha_{j}(x-r^{m_{0}}) r^{m_{j}}\Bigg).
	\end{equation*}
	
	We claim that $\|F\|_{U^{s}}^{2^{s}}\geq \delta^{C}N^{m_{0}(s+1)}$ implies $\|F\|_{U^{s-1}}^{2^{s-1}}\geq \delta^{C}N^{m_{0}s}$.
	Indeed, by the popularity principle and Lemma \ref{U^2 inverse}, for all $\boldsymbol{h}$ in some $\mathcal{H}\subset[2N^{m_{0}}]^{s-2}$ of density $\geq \delta^{C}$, there is $\phi=\phi(\boldsymbol{h})\in\mathbb{T}$ such that
	\begin{equation}\label{remove dependencies,eq2}
		\left|\underset{x\in[2N^{m_{0}}]}{\mathbf{E}}\Delta_{\boldsymbol{h}}F(x)\cdot e(x\phi(\boldsymbol{h}))\right|\geq \delta^{C}.
	\end{equation}
	At the cost of at most $\delta^{C}$, we may further assume that for all $\boldsymbol{h}\in \mathcal{H}$, the values of  $\phi(\boldsymbol{h})$ belong to the set $\left\{t/T_{0}:0\leq t<T_{0}\right\}$ for some $T_{0}=\lceil\delta^{-C} N^{m_{0}}\rceil$. We extend $\phi$ to a function on $\mathbb{Z}^{s-2}$ by setting it to be zero elsewhere. 	By using Lemma \ref{technical lemma 1}, a change of variables, and the triangle inequality, we get
	\begin{equation*}
		\underset{\boldsymbol{h^{0}}, \boldsymbol{h^{1}} \in \mathcal{H},x\in A}{\mathbf{E}}\left|\underset{r \in [N] }{\mathbf{E}} e\Bigg( \tilde{\phi}\left(\boldsymbol{h^{0}}, \boldsymbol{h^{1}}\right)r^{m_{0}} +\sum_{1 \leq j \leq k} \partial_{\boldsymbol{h^{0}}-\boldsymbol{h^{1}}}\alpha_{j}(x) r^{m_{j}}\Bigg)\right| \geq \delta^{C},
	\end{equation*}
	where
	\begin{equation*}
		\tilde{\phi}\left(\boldsymbol{h^{0}}, \boldsymbol{h^{1}}\right):=\sum_{\boldsymbol{\omega} \in\{0,1\}^{s-2}}(-1)^{|\boldsymbol{\omega}|} \phi\left(h_{1}^{\omega_{1}}, \ldots, h_{s-2}^{\omega_{s-2}}\right).
	\end{equation*}
	Hence, by the popularity principle and Lemma \ref{weyl}, for all $\boldsymbol{h^{0}}$ in some $\mathcal{H}^{\prime}\subset\mathcal{H}$ of density $\geq\delta^{C}$ and for some $\boldsymbol{h^{1}}\in\mathcal{H}$, there exists $1\leq q\leq \delta^{-C}$ such that
	\begin{equation*}
		\left\|q\tilde{\phi}\left(\boldsymbol{h^{0}}, \boldsymbol{h^{1}}\right)\right\|\leq \frac{\delta^{-C}}{N^{m_{0}}}\leq\frac{\delta^{-C}}{T_{0}}.
	\end{equation*}
	Since $\tilde{\phi}(\boldsymbol{h^{0}}, \boldsymbol{h^{1}})\in\mathbb{T}$ also takes values in $\left\{t/T_{0}:0\leq t<T_{0}\right\}$, we deduce from the above that the set $\{\tilde{\phi}(\boldsymbol{h^{0}}, \boldsymbol{h^{1}}): \boldsymbol{h^{0}}\in \mathcal{H}^{\prime}\}$ contains at most $\delta^{-C}$ different values. Hence $\tilde{\phi}(\cdot, \boldsymbol{h^{1}})$ is constant on some $\mathcal{H^{\prime\prime}}\subset \mathcal{H^{\prime}}$ of density $\geq\delta^{C}$. Consequently, $\phi$ has ``low rank'' on $\mathcal{H^{\prime\prime}}$. That is, there exist $\psi_{1},\ldots,\psi_{s-2}:\mathbb{Z}^{s-3}\rightarrow\mathbb{C}$ such that
	\begin{equation*}
		\phi(\boldsymbol{h})=\sum_{1\leq i\leq s-2}\psi_{i}\left(\widehat{\boldsymbol{h}}^{i}\right),\quad \boldsymbol{h}\in\mathcal{H^{\prime\prime}}.
	\end{equation*}
	The claim thus follows by summing both sides of \eqref{remove dependencies,eq2} over $\boldsymbol{h}\in\mathcal{H^{\prime\prime}}$  and applying Lemma \ref{technical lemma 2}.
		
	By iterating the claim proved above, we   obtain that $\|F\|_{U^{2}}^{4}$ is large. Furthermore, applying Lemma \ref{U^2 inverse}, a change of variables, and the triangle inequality yields
	\begin{equation*}
		\underset{x\in A}{\mathbf{E}}\left|\underset{r\in[N]}{\mathbf{E}}e\Bigg(\beta r^{m_{0}}+\sum_{1 \leq j \leq k} \alpha_{j}(x) r^{m_{j}}\Bigg)\right|\geq \delta^{C}
	\end{equation*}
	for some $\beta\in\mathbb{T}$. As before, we deduce that all $\alpha_{1},\ldots,\alpha_{k}$ are constant on some $A^{\prime}\subset A$ of density $\geq \delta^{C}$. The proof is then completed by summing both sides of \eqref{remove dependencies,eq1} over $x\in A^{\prime}$.
\end{proof}

To complete the proof of the case $n=1$ of Theorem \ref{Lambda_{N}-inverse with phase}, we need further information about the function $f$ beyond what is provided by Proposition \ref{remove dependencies}. This is given in the following proposition. While this result can be derived from Proposition \ref{remove dependencies} via standard Fourier methods (see, e.g., \cite[Lemma 5.2]{PPS24}), we present here an alternative proof that is closer in spirit to the approach taken in \cite[Theorem 6.57]{KMPW24}.

\begin{proposition}\label{standard Fourier method}
	Given $n\geq 1$ and a polynomial $P\in \mathbb{T}[r]$ that contains no term of degree $n$, there exists a constant $C=C(n,\deg (P))>0$ such that the following holds. 	
	Let $\delta\in (0,1/10)$, $N\geq 1$, and   $f:\mathbb{Z}\rightarrow\mathbb{C}$ be a 1-bounded function supported in $\left[2N^{n}\right]$. If $N\geq \delta^{-C}$ and
	\begin{equation}\label{standard Fourier method,eq1}
		\underset{x \in\left[N^{n}\right]}{\mathbf{E}}\left|\underset{r \in[N]}{\mathbf{E}} f\left(x+r^{n}\right) e(P(r))\right| \geq \delta,
	\end{equation}
	then there exists $1\leq q \leq \delta^{-C}$ such that for any $1\leq L \leq \delta^{C} N^{n}$,
	\begin{equation*}
		\left\|\mathbf{E}\left(\left.f \right\rvert \mathcal{B}_{(q L, q)}\right)\right\|_{2}^{2} \geq \delta^{C} N^{n}.
	\end{equation*}
\end{proposition}
\begin{proof}
	Denote
	\begin{equation*}
		F_N:=\left\{\xi\in\mathbb{T}:\left|\underset{r \in [N]}{\mathbf{E}} e\left(P(r)+\xi r^{n}\right) \right|\geq \delta/4\right\}.
	\end{equation*}
	If $N$ is large, then, by Lemma \ref{weyl}, we have
	\begin{equation}\label{standard Fourier method,eq2}
		F_N\subset \left\{\xi\in\mathbb{T}:\left\|q \xi \right\| \leq \delta^{-C}N^{-n}  \text{ for some } 1\leq q\leq \delta^{-C}\right\},
	\end{equation}
	where $C$ depends only on $n$ and $\deg(P)$. Especially, the measure of $F_N$ is at most $2\delta^{-2C}N^{-n}$.

 	Rewriting  \eqref{standard Fourier method,eq1} as
	\begin{equation*}
		\underset{x \in\left[N^{n}\right],r\in [N]}{\mathbf{E}} (\psi\chi_{\left[N^{n}\right]})(x) f\left(x+r^{n}\right) e(P(r)) \geq \delta
	\end{equation*}
	for some 1-bounded factor $\psi$, and applying the Fourier inversion, the triangle inequality, H\"{o}lder's inequality, Plancherel's identity, and the bound of $|F_N|$, we obtain
\begin{align*}
\delta N^{n}&\leq \int_{0}^{1}\left|\left(\psi\chi_{\left[N^{n}\right]}\right)^{\wedge}(-\xi)\right|\left|\hat{f}(\xi)\right| \left|\underset{r \in [N]}{\mathbf{E}} e\left(P(r)+\xi r^{n}\right)\right| \,\mathrm{d} \xi \\
	&\leq\frac{\delta}{4}\int_{F_N^c}\left|\left(\psi\chi_{\left[N^{n}\right]}\right)^{\wedge}(-\xi)\right|\left|\hat{f}(\xi)\right| \,\mathrm{d} \xi +\int_{F_N}\left|\left(\psi\chi_{\left[N^{n}\right]}\right)^{\wedge}(-\xi)\right|\left|\hat{f}(\xi)\right| \,\mathrm{d} \xi \\
			&\leq\frac{1}{2}\delta N^{n}+ 2\delta^{-2C} \sup_{\xi\in F_N}|\hat{f}(\xi)|.
\end{align*}
Thus, the set $F_N$ is nonempty, and for some $\xi_{0} \in F_N$,
	\begin{equation*}
		\left|\hat{f}(\xi_{0})\right|\geq \frac{1}{8}\delta^{2C+1}N^{n}.
	\end{equation*}
	By \eqref{standard Fourier method,eq2}, we have $\left\|q \xi_{0} \right\| \leq \delta^{-C}N^{-n}$ for some $1\leq q\leq \delta^{-C}$. Hence, for any $1\leq L\leq \delta^{10C+10}N^{n}$,
	\begin{align*}
			\left|\left(f-\mathbf{E}\left(\left.f\right\rvert\mathcal{B}_{(qL,q)}\right)\right)^{\wedge}(\xi_{0})\right|&=\left|\left\langle f, e(\xi_{0}\cdot)-\mathbf{E}(\left.e(\xi_{0}\cdot)\right\rvert \mathcal{B}_{(q L, q)})\right\rangle\right|\\
			&\leq 2N^{n}\left\|e(\xi_{0}\cdot)-\mathbf{E}(\left.e(\xi_{0}\cdot)\right\rvert \mathcal{B}_{(q L, q)})\right\|_{\infty}\\
			&\leq  \frac{1}{16}\delta^{2C+1}N^{n}.
	\end{align*}
	Therefore,
	\begin{equation*}	
		\frac{1}{16}\delta^{2C+1}N^{n}\leq \left|\mathbf{E}\left(\left.f\right\rvert\mathcal{B}_{(qL,q)}\right)^{\wedge}(\xi_{0})\right|\leq \left\|\mathbf{E}\left(\left.f\right\rvert\mathcal{B}_{(qL,q)}\right)\right\|_{1}.
	\end{equation*}
	The result then follows easily.
\end{proof}


\section{Further directions}\label{appendiB}

In this section we discuss possible generalizations to multidimensional polynomial progressions.

Recall that in the proof of our Szemer\'{e}di theorem, as stated right after Theorem \ref{Lambda_{N}-inverse}, we first use the inverse theorem (Theorem \ref{Lambda_{N}-inverse}) to derive, via an averaging process, an inverse theorem for the more general operator $\Lambda_{q, M, \boldsymbol{N};\boldsymbol{m}}$ (Theorem \ref{Lambda_{q,M}-inverse}), and then carry out an energy increment argument to handle $\Lambda_{q, M, \boldsymbol{N};\boldsymbol{m}}$ and complete the remaining steps. We note that Theorem $\ref{Lambda_{N}-inverse}$ is actually a special case of \cite[Theorem 4.14]{KMPWW24}, which already applies to general multidimensional polynomial progressions. However, the difficulty in extending our result to the general setting lies in the fact that the argument in Section \ref{section:strategy} heavily relies on homogeneity to obtain the inverse theorem for $\Lambda_{q,M,\boldsymbol{N};\boldsymbol{m}}$; this step does not yet extend to the general case.

We now illustrate a possible approach to the inhomogeneous problem by means of the following example
\begin{equation}\label{further direction,eq1}
	(x,y),(x+r^{2}+r,y),(x,y+r^{3}).
\end{equation}
A natural way  to deal with this inhomogeneous configuration is to first consider its homogeneous three‑dimensional counterpart
\begin{equation}\label{further direction,eq2}
	(x,y,z),(x+r,y+r^{2},z),(x,y,z+r^{3}).
\end{equation}
A simple projection argument, analogous to that used in the proof of Corollary \ref{popular difference,1-D}, would then allow us to return to the original inhomogeneous configuration \eqref{further direction,eq1}.

The study of \eqref{further direction,eq2} using the methods of this paper reduces to a problem similar to that in Proposition \ref{remove dependencies}. Specifically, suppose that
\begin{equation} \label{s7-1}
	\underset{x\in[N],y\in\left[N^{2}\right]}{\mathbf{E}}\left|\underset{r \in[N]}{\mathbf{E}} f\left(x+r,y+r^{2}\right) e\left(\alpha(x,y)r^{3}\right)\right| \geq \delta.
\end{equation}
Can we find a $\beta\in \mathbb{T}$, independent of $x$ and $y$, such that
\begin{equation}\label{further direction,eq3}
	\underset{x\in[N],y\in\left[N^{2}\right]}{\mathbf{E}}\left|\underset{r \in[N]}{\mathbf{E}} f\left(x+r,y+r^{2}\right) e\left(\beta r^{3}\right)\right|
\end{equation}
is also large?  If such a $\beta$ exists, one can either use the standard Fourier method to show that some Fourier coefficient $|\hat{f}(\xi,\eta)|$ must be large, or use the method in the proof of Proposition \ref{standard Fourier method} to show that
\begin{equation*}
	\left\|\mathbf{E}\left(\left.f \right\rvert \mathcal{B}_{(q L, q)} \otimes \mathcal{B}_{\left(q^{2} L^{2}, q^{2}\right)}\right)\right\|_{2}^{2}
\end{equation*}
is large for some $q$ and any small $L$; here $\mathcal{B}_{(q L, q)} \otimes \mathcal{B}_{\left(q^{2} L^{2}, q^{2}\right)}$ denotes a partition of $\mathbb{Z}^{2}$. With this in hand, the rest of the argument follows smoothly.

We believe that the question \eqref{further direction,eq3} should admit a positive answer. This belief is motivated by the following heuristic. Under the additional assumption that $|f|\equiv1$ on $[N]\times[N^{2}]$, we may write $f=e(\psi)$ and expand $\psi$ via a ``discrete Taylor series'' to obtain
\begin{equation}\label{further direction,eq4}
	\underset{x \in[N], y \in\left[N^{2}\right]}{\mathbf{E}}\Bigg|\underset{r \in[N]}{\mathbf{E}} e\Bigg(\sum_{i, j \geq 0} \gamma_{i, j}(x+r)^{i}\left(y+r^{2}\right)^{j}+\alpha(x, y) r^{3}\Bigg)\Bigg| \geq \delta.
\end{equation}
Terms of higher order should be eliminable by Lemma \ref{weyl}. One would then expect to obtain an estimate of the form
\begin{equation*}
	\left\|q\left(\gamma_{3,0}+\gamma_{1,1}+\alpha(x,y)\right)\right\|\lesssim_{\delta}N^{-3},
\end{equation*}
which would allow us to make $\alpha(x,y)$ constant after restricting $(x,y)$ to some suitable  smaller set.

However, the methods developed in this paper do not suffice to obtain results of the form \eqref{further direction,eq4} from the assumption \eqref{s7-1}. It seems that we always obtain ``bad'' expressions, for which no “reasonable” change of variable in $x$ is available. Thus, to address the question \eqref{further direction,eq3}, new ideas may be needed.


\subsection*{Acknowledgments}
Our thanks go to Prof. Ben Green, Prof. Sarah Peluse, and Prof. Sean Prendiville for answering our questions and clearing up our confusion. We would also like to thank James Leng and Xuezhi Chen for helpful communication.



\begin{thebibliography}{00}


\bibitem{BL96}
Bergelson, V., Leibman, A.,
Polynomial extensions of van der Waerden's and Szemer\'{e}di's theorems,
\textit{J. Amer. Math. Soc.} \textbf{9} (1996), no. 3, 725--753.


\bibitem{BS23}
Bloom, T. F.,  Sisask, O.,
An improvement to the Kelley--Meka bounds on three-term arithmetic progressions,
preprint, arXiv:2309.02353.

\bibitem{CG24}
Chen, X., Guo, J.,
A polynomial Roth theorem for corners in $\mathbb{R}^2$ and a related bilinear singular integral operator,
\textit{Math. Ann.} \textbf{390} (2024), no. 1, 255--301.

\bibitem{CM24}
Chen, X.,  Miao, C.,
Two-point polynomial patterns in subsets of positive density in $\mathbb{R}^n$,
\textit{Int. Math. Res. Not. IMRN} \textbf{2024}, no. 14, 10865--10879.


\bibitem{CDR21}
Christ, M., Durcik, P.,  Roos, J.,
Trilinear smoothing inequalities and a variant of the triangular Hilbert transform,
\textit{Adv. Math.} \textbf{390} (2021), Paper No. 107863, 60 pp.


\bibitem{FK78}
Furstenberg, H., Katznelson, Y.,
An ergodic Szemer\'{e}di theorem for commuting transformations,
\textit{J. Analyse Math.} \textbf{34} (1978), 275--291.



\bibitem{G98}
Gowers, W. T.,
A new proof of Szemer\'{e}di's theorem for arithmetic progressions of length four,
\textit{Geom. Funct. Anal. }\textbf{8} (1998), no. 3, 529--551.

\bibitem{G01}
Gowers, W. T.,
A new proof of Szemer\'{e}di's theorem,
\textit{Geom. Funct. Anal.} \textbf{11} (2001), no. 3, 465--588.



\bibitem{GT09}
Green, B., Tao, T.,
New bounds for Szemer\'{e}di’s theorem. II. A new bound
for $r_4(N)$,
\textit{Analytic number theory}, 180--204, Cambridge  University  Press,
Cambridge, 2009.

\bibitem{GT17}
Green, B., Tao, T.,
New bounds for Szemer\'{e}di’s theorem, III: a polylogarithmic bound for $r_4(N)$,
\textit{Mathematika} \textbf{63} (2017), no. 3, 944--1040.


\bibitem{HLY21}
Han, R., Lacey, M. T., Yang, F.,
A polynomial Roth theorem for corners in finite fields,
\textit{Mathematika} \textbf{67} (2021), no. 4, 885--896.

\bibitem{KM23}
Kelley, Z., Meka, R.,
Strong bounds for 3-progressions,
2023 IEEE 64th Annual Symposium on Foundations of Computer Science---FOCS 2023, 933--973.
IEEE Computer Society, Los Alamitos, CA, 2023.

\bibitem{KMPWW24}
Kosz, D., Mirek, M., Peluse S., Wan, R., Wright, J.,
The multilinear circle method and a question of Bergelson,
preprint, arXiv:2411.09478v3.


\bibitem{KMPW24}
Krause, B.,  Mirek, M., Peluse, S., Wright, J.,
Polynomial progressions in topological fields,
\textit{Forum Math. Sigma} \textbf{12} (2024), Paper No. e106, 51 pp.


\bibitem{KKL24 b}
Kravitz, N., Kuca, B.,  Leng, J.,
Quantitative concatenation for polynomial box norms,
\textit{Adv. Math.} \textbf{489} (2026), Paper No. 110820, 82 pp.


\bibitem{KKL24 a}
Kravitz, N., Kuca, B.,  Leng, J., Corners with polynomial side length, preprint, arXiv:2407.08637v2.


\bibitem{Kuc24}
Kuca, B.,
Multidimensional polynomial Szemer\'{e}di theorem in finite fields for polynomials of distinct degrees,
\textit{Israel J. Math.} \textbf{259} (2024), no. 2, 589--620.


\bibitem{Kuc24-b}
Kuca, B.,
Multidimensional polynomial patterns over finite fields: bounds, counting estimates and Gowers norm control,
\textit{Adv. Math.} \textbf{448} (2024), Paper No. 109700, 61 pp.



\bibitem{LSS24-1}
Leng, J., Sah, A.,  Sawhney, M.,
Improved bounds for five-term arithmetic progressions,
\textit{Math. Proc. Cambridge Philos. Soc.} \textbf{177} (2024), no. 3, 371--413.


\bibitem{LSS24-2}
Leng, J., Sah, A.,  Sawhney, M.,
Improved bounds for Szemer\'{e}di's theorem, preprint,
arXiv:2402.17995.

\bibitem{Pel19}
Peluse, S.,
On the polynomial Szemer\'{e}di theorem in finite fields,
\textit{Duke Math. J.} \textbf{168} (2019), no. 5, 749--774.

\bibitem{Pel20}
Peluse, S.,
Bounds for sets with no polynomial progressions,
\textit{Forum Math. Pi} \textbf{8} (2020), e16, 55 pp.

\bibitem{PP24}
Peluse, S., Prendiville, S.,
Quantitative bounds in the nonlinear Roth theorem,
\textit{Invent. Math. }
\textbf{238} (2024), no. 3, 865--903.

\bibitem{PP22}
Peluse, S., Prendiville, S.,
A polylogarithmic bound in the nonlinear Roth theorem,
\textit{Int. Math. Res. Not. IMRN} \textbf{2022}, no. 8, 5658--5684.



\bibitem{PPS24}
Peluse, S., Prendiville, S., Shao, X.,
Bounds in a popular multidimensional nonlinear Roth theorem,
\textit{J. Lond. Math. Soc. (2)} \textbf{110} (2024), no. 5, Paper No. e70019, 35 pp.


\bibitem{Pre17}
Prendiville, S.,
Quantitative bounds in the polynomial Szemer\'{e}di theorem: the homogeneous
case,
\textit{Discrete Anal.} 2017, Paper No. 5, 34 pp.


\bibitem{R54}
Roth, K. F.,
On certain sets of integers. II,
\textit{J. London Math. Soc.} \textbf{29} (1954), 20--26.


\bibitem{S75}
Szemer\'{e}di, E.,
On sets of integers containing no $k$ elements in arithmetic progression,
\textit{Acta Arith.} \textbf{27} (1975), 199--245.


\bibitem{SW25}
Shao, X., Wang, M.,
Quantitative bounds in a popular polynomial Szemer\'{e}di theorem, 
\textit{Proc. Roy. Soc. Edinburgh Sect. A}, Published online 2025:1--27.





\bibitem{Tao12}
Tao, T.,
Higher order Fourier analysis,
Grad. Stud. Math., 142
American Mathematical Society, Providence, RI, 2012. x+187 pp.





\end{thebibliography}
\end{document}